\documentclass[12pt]{amsart}

\usepackage[margin = 1in]{geometry}
\usepackage{amsfonts, amsmath,amscd, amssymb, cite, color, latexsym, mathrsfs, 
slashed, stmaryrd, verbatim, wasysym , tensor }
\usepackage[all]{xy}
\usepackage{graphicx}
\usepackage{hyperref}
\usepackage[applemac]{inputenc}

 \usepackage{tikz}

\newtheorem*{acknowledgements}{Acknowledgements}
\newtheorem{theorem}{Theorem}[section]
\newtheorem{corollary}[theorem]{Corollary}
\newtheorem{definition}[theorem]{Definition}

\newtheorem{lemma}[theorem]{Lemma}
\newtheorem{proposition}[theorem]{Proposition}
\newtheorem{remark}[theorem]{Remark}
\newtheorem{assumption}[theorem]{Assumption}

\numberwithin{equation}{section}

\let\oldsqrt\sqrt
\def\sqrt{\mathpalette\DHLhksqrt}
\def\DHLhksqrt#1#2{%
\setbox0=\hbox{$#1\oldsqrt{#2\,}$}\dimen0=\ht0
\advance\dimen0-0.2\ht0
\setbox2=\hbox{\vrule height\ht0 depth -\dimen0}%
{\box0\lower0.4pt\box2}}

\newcommand{\bhs}[1]{\mathfrak B_{#1}}

\renewcommand{\Re}{\operatorname{Re}}
\renewcommand{\hat}[1]{\widehat{#1}}

\newcommand\pa{\partial}

\newcommand\cf{cf\@. }

\newcommand\eg{e\@.g\@. }

\newcommand\eps\varepsilon
\renewcommand\epsilon\varepsilon

\newcommand\ephi{\operatorname{\eps, \phi}}
\newcommand\ed{\operatorname{\eps, d}}

\newcommand\Ed{{}^{\ed}}

\newcommand\CI{{\mathcal{C}}^{\infty}}

\newcommand\ang[1]{\langle #1 \rangle}

\renewcommand\det{\operatorname{det}}

\newcommand\Diff{\operatorname{Diff}}

\newcommand\End{\operatorname{End}}

\DeclareMathOperator*{\FP}{\operatorname{FP}}

\newcommand\Id{\operatorname{Id}}
\renewcommand\Im{\operatorname{Im}}

\newcommand\phg{\operatorname{phg}}

\newcommand\rank{\operatorname{rank}}

\renewcommand\Re{\operatorname{Re}}

\newcommand\sign{\operatorname{sign}}

\newcommand\Tr{\operatorname{Tr}}

\newcommand\pr{\operatorname{pr}}

\newcommand\GL{\operatorname{GL}}

\newcommand\spane{\operatorname{span}}
\newcommand\Ad{\operatorname{Ad}}

\newcommand\sma{\operatorname{small}}

\newcommand\paperintro%
        {%
         }
\newcommand\paperbody%
        {%
         }

\newcommand\bbC{\mathbb{C}}

\newcommand\bbH{\mathbb{H}}

\newcommand\bbN{\mathbb{N}}

\newcommand\bbR{\mathbb{R}}
\newcommand\bbS{\mathbb{S}}

\newcommand\bbZ{\mathbb{Z}}

\newcommand\cA{\mathcal{A}}

\newcommand\cH{\mathcal{H}}

\newcommand\cK{\mathcal{K}}

\newcommand\cU{\mathcal{U}}
\newcommand\cV{\mathcal{V}}

\newcommand\SL{\operatorname{SL}}
\newcommand\SO{\operatorname{SO}}
\newcommand\Sym{\operatorname{Sym}}
\newcommand\Spin{\operatorname{Spin}}
\newcommand\bX{\overline{X}}

\newcommand\tE{\widetilde{E}}
\newcommand\bx{\overline{x}}
\newcommand\SU{\operatorname{SU}}
\newcommand\bz{\overline{z}}
\newcommand\vol{\operatorname{vol}}

\newcommand\WH{\operatorname{WH}}
\newcommand {\Eis}{\operatorname{Eis}}

\DeclareMathAlphabet{\mathpzc}{OT1}{pzc}{m}{it}

\newcommand\hyp{\operatorname{hyp}} 

\begin{document}

\title[Analytic torsion and Reidemeister torsion]{Analytic torsion and Reidemeister torsion of hyperbolic manifolds with cusps}

\author{Werner Müller}
\address{Universität Bonn, Mathematisches Institut, Endnicher Allee 60, D-53115 Bonn, Germany}
\email{mueller@math.uni-bonn.de}
\author{Fr\'ed\'eric Rochon}
\address{D\'epartement de Math\'ematiques, UQ\`AM}
\email{rochon.frederic@uqam.ca }

\maketitle

\begin{abstract}
On an odd-dimensional oriented hyperbolic manifold of finite volume with strongly acyclic coefficient systems, we derive a formula relating analytic torsion with the Reidemeister torsion of the Borel-Serre compactification of the manifold.  In a companion paper, this formula is used to derive exponential growth of torsion in cohomology of arithmetic groups.  
\end{abstract}

\tableofcontents

\section{Introduction}

Let $M$ be a closed Riemannian manifold of dimension $d$, and $\varrho$ a finite
dimensional complex representation of the fundamental group $\pi_1(M,x_0)$ of
$M$. Let $E_\varrho\to X$ be the flat vector bundle over $X$ associated to 
$\varrho$.
Choose a Hermitian fiber metric in $E_\varrho$ and let $\Delta_p(\varrho)$ be 
the Laplace operator on $E_\varrho$-valued $p$-forms with respect to the metric 
on $X$ and in $E_\varrho$. Let $\zeta_p(s;\varrho)$ be the zeta function of 
$\Delta_p(\varrho)$ (see \cite{Sh}). Then the analytic torsion 
$T_X(\varrho)\in\bbR^+$, introduced by Ray and Singer \cite{Ray-Singer}, is 
defined by
\begin{equation}\label{tor}
\log T_X(\varrho):=\frac{1}{2}\sum_{q=1}^d (-1)^q q\frac{d}{ds}
\zeta_q(s;\varrho)\big|_{s=0}.
\end{equation}
A combinatorial counterpart is the Reidemeister torsion introduced by 
Reidemeister \cite{Reidemeister1935} and Franz \cite{Franz1935} to distinguish 
lens spaces that are homotopic but not homeomorphic. It was conjectured by
Ray and Singer and proved independently by Cheeger \cite{Cheeger1979} and the 
first named author \cite{Muller1978} that for unitary representations the
invariants coincide. The equality was extended by the first author to unimodular
representations \cite{Muller1993}. The case of a general representation was 
treated by Bismut and Zhang \cite{Bismut-Zhang}. In general the equality does 
not hold. The defect was computed by Bismut and Zhang. 

The equality of analytic and Reidemeister torsion has recently been used to 
study the growth of torsion in the cohomology of arithmetic groups, 
see for instance \cite{BV, Calegari-Venkatesh, Muller2012, Marshall-Muller, 
MP2014b}. This application is based on a remarkable feature of the
Reidemeister torsion, and hence of the analytic torsion. When the complex of 
cochains, used to define the Reidemeister torsion, is defined over 
$\bbZ$, for instance when $\varrho$ is the trivial representation, then the 
Reidemeister torsion can be expressed in terms of the size of the torsion
subgroup of the integer cohomology and the covolume of the lattice  defined
by the free part in the real cohomology. For various sequences of manifolds or 
representations, this can be used to establish exponential growth of torsion 
subgroups in cohomology by computing the limiting behavior of analytic torsion 
via spectral methods. 

In the context of arithmetic groups the manifolds are 
compact locally symmetric manifolds $\Gamma\backslash G/K$, where $G$ is a 
semi-simple Lie group, $K$ maximal compact subgroup and $\Gamma$ a discrete 
torsion free cocompact subgroup of $G$.

Since many arithmetic groups are not cocompact, it is very desirable to extend
this method to the non-compact case. The goal of the present paper is to study 
the relation between (regularized) analytic torsion and Reidemeister torsion
for odd-dimensional hyperbolic manifolds of finite volume. 

There is related work in \cite{ARS1}.
For odd-dimensional manifolds with fibered cusps ends, which is an important class of complete non-compact Riemannian manifolds of finite volume including many examples of locally symmetric spaces of rank one, an identification of analytic torsion with the Reidemeister torsion of  the natural compactification by a manifold with boundary was obtained in \cite{ARS1} provided that the unimodular representation $\varrho: \pi_1(M)\to \GL(V)$ is `acyclic at infinity' in a certain sense,  that the links of the cusps at infinity are even-dimensional, and that the Hermitian metric of the flat vector bundle associated to $\varrho$ is even in the sense of \cite[Definition~7.6]{ARS1}, the latter condition being automatically satisfied when the representation $\varrho$ is unitary.  The results of \cite{ARS1} apply in particular to odd-dimensional oriented hyperbolic manifolds of finite volume. However, they do not apply to the representations $\varrho$ that 
we wish to consider and which are described as follows.

Let $G=\SO_0(d,1)$ and $K=\SO(d)$ or $G=\Spin(d,1)$ and $K=\Spin(d)$. Then 
$G/K$, equipped with the normalized invariant metric, is isometric 
to the $d$-dimensional hyperbolic space $\bbH^d$.  Let $\Gamma\subset G$ be 
a torsion free lattice in $G$. Then 
\[
      X:= \Gamma\setminus G/K
\]
is an oriented $d$-dimensional hyperbolic manifold of finite volume whose 
hyperbolic metric will be denoted by $g_X$. Let $\varrho: G\to \GL(V)$ be an 
irreducible finite dimensional complex representation such that
\begin{equation}
  \varrho\circ \vartheta \ne \varrho,
\label{int.1}\end{equation}
where $\vartheta$ is the standard Cartan involution with respect to $K$.  By \cite[Lemma~4.3]{Muller1993}, the restriction of $\varrho$ to $\Gamma$ induces a unimodular  representation $\left. \varrho\right|_{\Gamma}: \Gamma\to \GL(V),$ where $\Gamma$ is identified with $\pi_1(X)$.  If $E=\bbH^d\times_{\left.  \varrho\right|_{\Gamma}}V$ is the associated flat vector bundle on $X$, then we know from \cite{MM1963} that $E$ comes equipped with a natural Hermitian metric $h_E$ well-defined up to a scalar multiple.  Notice however that this Hermitian metric is definitely not even in the sense of \cite[Definition~7.6]{ARS1}.  In fact, the Hermitian metric $h_E$ degenerates at infinity, so the identification between analytic and Reidemeister torsions obtained in \cite{ARS1} does not apply.  

On the other hand, using a different approach relying on the gluing formula of Lesch \cite{Lesch2013},  Pfaff was able in \cite{Pfaff2017} to obtain a formula relating the analytic torsion of $(X,E)$ with the Reidemeister torsion of the natural compactification of $X$ by a manifold with boundary $\bX$.  One delicate point in the formula is that $(X,g_X,E,h_E)$ always has trivial $L^2$-cohomology, but the cohomology groups $H^q(\bX;E)$ are not all trivial.  Thus, to define Reidemeister torsion, one has to specify a basis of these cohomology groups.  Pfaff does it using Eisenstein series.  When $G=\Spin(3,1)\cong \SL(2,\bbC)$ and $\Gamma$ is a congruent subgroup of a Bianchi group, the formula of \cite{Pfaff2017} was subsequently used by Pfaff and Raimbault \cite{Pfaff-Raimbault} to obtain results about exponential growth of torsion in cohomology for the sequence of symmetric powers of the standard representation of $\SL(2,\bbC)$.  However, beyond that, the formula of \cite{Pfaff2017} contains a term, namely the analytic torsion of the cusp ends, which so far seems to have restricted the possible applications about the growth of torsion in cohomology.  

In the present paper, we remedy this problem by obtaining a formula where the 
analytic torsion of the cusp ends does not appear.  To state our result more 
precisely, recall that the analytic torsion $T(X;E, g_X,h_E)$ of 
$(X,g_X,E, h_E)$ is defined by the following formula which is analogous to
 \eqref{tor}
\begin{equation}
\log T(X;E,g_X,h_E)= \frac12 \sum_{q=0}^d (-1)^q q 
\frac{d}{ds}{}^{R}\zeta_q(s;\varrho)\big|_{s=0},
\label{int.2}\end{equation}
where  ${}^{R}\zeta_q(s;\varrho)$ is the regularized zeta function of the
Laplace operator $\Delta_q(\varrho)$ acting on $E$-valued $q$-forms.
For $\Re s> \frac{d}2$, this function is defined by 
\begin{equation}
{}^{R}\zeta_q(s;\varrho) = \frac{1}{\Gamma(s)} \int_{0}^{\infty} 
{}^{R}\Tr(e^{-t\Delta_q(\varrho)} ) t^s \frac{dt}t
\label{int.3}\end{equation}    
 and admits a meromorphic extension which is regular at $s=0$, where ${}^{R}\Tr$ is the regularized trace as considered in \cite{ARS1} or \cite{MP2012}. 
 
In order to define the Reidemeister torsion, we need to choose a basis
of $H^*(\bX,E)$. For the flat bundles that we consider, the cohomology 
$H^*(\bX,E)$ never vanishes. More precisely, the $L^2$-cohomology 
$H^*_{(2)}(X,E)$ vanishes (see \cite[Prop. 8.1]{Pfaff2017}), which corresponds to the 
vanishing of the cohomology in the compact case.  However, there is cohomology 
coming from the boundary of $\overline X$. This is
the Eisenstein cohomology $H^\ast_{\Eis}(X,E)$ introduced by Harder \cite{Harder}.
In our case $H^*(\bX,E)$ coincides with $H^*_{\Eis}(X,E)$ and each cohomology
class is represented by a special value of an Eisenstein series, which is a lift
of a cohomology class on the boundary $Z$. Using an orthonormal basis 
$\mu_Z$ of $H^*(Z,E)$, the theory of Eisenstein series gives rise to a basis 
$\mu_X$ of $H^*(\bX,E)$. For more details see \cite[sect. 8]{Pfaff2017} and \S5 below.
These are the bases of the cohomology that we use to define the Reidemeister
torsion $\tau(\overline X,E,\mu_X)$ and $\tau(Z,E,\mu_Z)$.

 We can now state our main result, referring to Theorem~\ref{cm.11c} below for further details.
 \begin{theorem}
 If the complex irreducible representation $\varrho:G\to \GL(V)$ satisfies \eqref{int.1} and if the discrete subgroup $\Gamma\subset G$ is such that Assumption~\ref{dc.3b} below holds, then 
\begin{equation}
 \log T(X;E,g_X,h_E)= \log \tau(\bX,E,\mu_X)- \frac12 \log \tau(Z,E,\mu_Z)- \kappa^{\varrho}_{\Gamma} c_{\varrho},
\label{int.3d}\end{equation}
where $c_{\varrho}\in \bbR$ is an explicit constant depending on $\varrho$, $\kappa^{\varrho}_{\Gamma}$ is the number of connected components of $Z$ on which the cohomology with values in $E$ is non-trivial, and $\mu_X$ and $\mu_Z$ are cohomology bases of $H^*(\bX;E)$ and $H^*(Z;E)$ described above (\cf \S~\ref{be.0} below).  Furthermore, if $n$ is odd, then 
 $\tau(Z,E,\mu_Z)=1$ and the formula simplifies to
 $$
 \log T(X;E,g_X,h_E)= \log \tau(\bX,E,\mu_X)- \kappa^{\varrho}_{\Gamma} c_{\varrho}.
$$
  \label{int.4}\end{theorem} 
\begin{remark}
In a companion paper \cite{MR2}, this formula is used to establish exponential growth of torsion in cohomology for various sequences of groups $\Gamma$ or representations $\varrho$.
\end{remark}

\begin{remark}
 When $G=\Spin(d,1)$, Assumption~\ref{dc.3b} is slightly more general then what the hypothesis \cite[(2.11)]{Pfaff2017} requires for the result of Pfaff.  This allows in particular for situations where $\kappa^{\varrho}_{\Gamma}=0$, in which case $H^*(\bX;E)$ and $H^*(Z;E)$ are trivial and our formula simplifies to   
$$
 \log T(X;E,g_X,h_E)= \log \tau(\bX,E)- \frac12 \log \tau(Z,E).
$$
\label{int.5}\end{remark} 
\begin{remark}
When $G=\Spin(3,1)\cong \SL(2,\bbC)$  and $X=\Gamma\setminus G /K$ is the complement of a hyperbolic knot, we know from \cite{MFP2014} that $\kappa^{\varrho}_{\Gamma}=0$ when $\varrho$ is an odd symmetric power of the standard representation of $\SL(2,\bbC)$.  Since $n=1$ is odd in this case, this means that the formula simplifies to
$$
 T(X;E,g_X,h_E)= \tau(\bX,E).
$$
\label{int.6}\end{remark} 
 \begin{remark}
 When $G=\Spin(3,1)\cong \SL(2,\bbC)$ and $\varrho$ is an even symmetric power of the standard representation, our formula agrees with the one obtained by Pfaff in \cite{Pfaff2014} for normalized analytic and Reidemeister torsions and yields the identity 
 \begin{equation}
 \frac{c(\ell)}{c(2)}= \frac{b(\ell)}{b(2)} \quad \mbox{for} \; \ell\ge 2,
\label{int.6b}\end{equation}
where 
\begin{equation}
\begin{gathered}
  b(\ell):= \frac{1}{2\ell+2}\prod_{k=-\ell}^{\ell-1}\left( \frac{\sqrt{(\ell+1)^2+\ell^2-k^2}-k-1}{\sqrt{(\ell+1)^2+\ell^2-(k+1^2)}-k} \right)^{\frac12}, \\
   c(\ell):= \frac{\prod_{j=1}^{\ell-1} (\sqrt{(\ell+1)^2+\ell^2-j^2}+\ell)  }{\prod_{j=1}^{\ell} (\sqrt{(\ell+1)^2+\ell^2-j^2}+\ell+1) }\left( \frac{\sqrt{(\ell+1)^2+\ell^2}+\ell }{\sqrt{(\ell+1)^2+\ell^2}+\ell+1 } \right)^{\frac12}.
\end{gathered}
\label{int.6c}\end{equation}

 \label{int.7}\end{remark}
 
 \begin{remark}
Multiplying the boundary defining function by a constant changes the bases $\mu_Z$ and $\mu_{\bX}$ as well as the right hand side of  \eqref{int.3d}, which is consistent with the fact that analytic torsion does depend on the choice of boundary defining function used to  define the regularized trace.  
 \end{remark}
 
 Combined with \cite[Theorem~1.1]{MP2012}, our results yield the following corollary, proved at the end of \S~\ref{cm.0}, about the exponential growth of the Reidemeister torsion for certain sequences of representations.
 
 \begin{corollary}
 Assume that $G=\SO_0(d,1)$, that $n$ is odd and  that $\Gamma$ satisfies Assumption~\ref{dc.3b}.  Fix natural numbers $\tau_1\ge \tau_2\ge\cdots \ge \tau_{n+1}$.  For $m\in \bbN$, let $\tau(m)$ be the finite-dimensional irreducible representation of $G$ with highest weight $(\tau_1+m,\ldots, \tau_{n+1}+m)$ and denote by $E_{\tau(m)}$ the corresponding flat vector bundle on $X=\Gamma\setminus G/K$.  Let also $\mu_{X,m}$ be the corresponding basis of $H^*(\bX;E_{\tau(m)})$.  Then there is a constant $C_n>0$ depending only on $n$ such that 
 \begin{equation}
       \tau(\bX,E_{\tau(m)},\mu_{X,m})= C_n\vol(X) m^{\frac{n(n+1)}2+1} +\mathcal{O}(m^{\frac{n(n+1)}2}\log m) \quad \mbox{as} \; m\to \infty.
 \label{int.7c}\end{equation}
 \label{int.7b}\end{corollary}
 \begin{remark}
If in fact $n=1$ and $G=\SL(2,\bbC)$, a formula similar to \eqref{int.7c} was obtained by Menal-Ferrer and Porti in \cite{MFP2014} using \cite{Muller2012} and  suitable approximations of $X$ by compact hyperbolic manifolds.   
 \label{int.7d}\end{remark}
 
Our strategy to prove Theorem~\ref{int.4} is to apply the general approach of \cite{ARS1}.  Indeed, even if \cite[Theorem~1.3]{ARS1} does not apply since $h_E$ is not an even Hermitian metric, the results of \cite{ARS1} concerning the uniform constructions of the resolvent and heat kernel under a degeneration to fibered cusps are formulated quite generally and do apply.  This is because the Hermitian metric $h_E$ degenerates at infinity in a similar way that $g_X$ does, which ensures that this can be incorporated in the framework of \cite{ARS1}.  More precisely, if $M= \bX\cup_{\pa\bX}\bX$ is the double of $\bX$ on which we consider a family $g_{\epsilon}$ of Riemannian metrics degenerating to the hyperbolic metric on each copy of $X$ inside $M$ as $\epsilon \searrow 0$, then recall from \cite{ARS1} that in a tubular neighborhood $N\cong \pa \bX\times (-\delta,\delta)_x$, we can take  $g_{\epsilon}$  of the form
$$
      g_{\epsilon}= \frac{dx^2}{x^2+\epsilon^2}+ (x^2+\epsilon^2) g_{\pa \bX}, \quad x\in (-\delta,\delta),
$$ 
where $g_{\pa \bX}$ is the (flat) Riemannian metric on $\pa\bX$ such that 
$$
g_X= \frac{dx^2}{x^2}+ x^2 g_{\pa \bX}, \quad  x\in (0,\delta),
$$
outside a compact set of $X$.  For the Hermitian metric $h_{E}$, we can in a similar way introduce a family of Hermitian metrics $h_{\epsilon}$ on the double of $E$ on $M$ degenerating to the Hermitian metric $h_E$ on each copy of $X$ as $\epsilon\searrow 0$.  This can be described in a systematic way using the single surgery space of Mazzeo-Melrose \cite{mame1}.  The upshot is that we end up with a family of Dirac-type operators for which the uniform constructions of the resolvent and heat kernel of \cite[Theorem~4.5 and Theorem~7.1]{ARS1} do apply directly.  The way the Hermitian metric $h_E$ degenerates is then incorporated in the model operators $D_v$ and $D_b$ of \cite{ARS1}, \cf \cite[(2.7) and (2.8)]{ARS1} with \eqref{sm.15} and \eqref{sm.17} below.  Taking this into account, we can still show that the model operator $D_b$ is Fredholm, which ensures that the eigenspace of eigenvalues of the Hodge Laplacian going to zero as $\epsilon\searrow 0$ is finite dimensional.  Furthermore, comparing the cohomology of $M$ taking values in the double of $E$ with the $L^2$-kernel of $D_b$ allows us to conclude that there is a spectral gap: no positive eigenvalue tends to zero as $\epsilon\searrow 0$.  This greatly simplifies the computation of the asymptotic behavior of analytic torsion as $\epsilon\searrow 0$, since this amounts essentially to compute the analytic torsion of the operator $D^2_b$.  To do this, we rely on the delicate computation in \cite[\S~2.2]{ARS2} of regularized determinants of Laplace-type operators on the real line.  

For Reidemeister torsion, we can track relatively easily what happens under surgery using the formula of Milnor \cite{Milnor1966}.  What is more delicate however is that as in \cite[\S3.3]{ARS2}, we need to carefully compute what happens asymptotically to a basis of orthonormal harmonic forms as $\epsilon\searrow 0$.  One subtle point is that to understand what happens at the cohomological level, it is not enough to determine the top order behavior.  Indeed, there are lower other terms in the expansion which are negligible in terms of $L^2$-norm as $\epsilon\searrow 0$, but nevertheless contribute non-trivially cohomologically, a phenomenon intimately related with the behavior at infinity of the Eisenstein series used by Pfaff in \cite{Pfaff2017}.              
 
 The paper is organized as follows.  In \S~\ref{cb.0}, we recall basic properties of the canonical bundle $E$ associated to a choice of irreducible complex representation $\varrho$. We then describe in \S~\ref{sm.0} the cusp surgery metric and the corresponding degenerating family of Hermitian metrics and compute explicitly what are the model operators $D_v$ and $D_b$.  This is used in \S~\ref{dat.0}  to determine the asymptotic behavior of analytic torsion as $\epsilon\searrow 0$.  In \S~\ref{be.0}, we introduce the basis $\mu_X$ and $\mu_Z$ of Theorem~\ref{int.4} and give a formula relating the Reidemeister torsions of $M$ and $\bX$.  Finally, we prove our main result in \S~\ref{cm.0}, while in \S~\ref{ed3.0}, we compute more precisely the constant $c_{\varrho}$ when $G=\Spin(3,1)\cong \SL(2,\bbC)$.
 
\begin{acknowledgements}
The authors are grateful to the hospitality of the Centre International de Rencontres Mathématiques (CIRM) where this project started.  The second author was supported by NSERC and a Canada Research Chair. 
\end{acknowledgements}

\section{The canonical bundle of Matsushima and Murakami}  \label{cb.0}

Let $d=2n+1$ be odd and consider the $d$-dimensional hyperbolic space seen as the homogeneous space $\bbH^d= G/K$ with either $G=\Spin(d,1)$ and $K=\Spin(d)$ or 
$G=\SO_o(d,1)$ and $K=\SO(d)$.  The hyperbolic metric $g_{\hyp}$ on $\bbH^d$ can be described in terms of the Killing form $B$ of the Lie algebra $\mathfrak{g}$ of $G$.
Indeed, if $\mathfrak{k}$ is the Lie algebra of $K$, $\vartheta$ is the standard Cartan involution with respect to $K$ and $\mathfrak{g}= \mathfrak{k}\oplus \mathfrak{p}$ is the Cartan decomposition of $\mathfrak{g}$, then the restriction of 
\begin{equation}
   \langle X,Y \rangle_{\vartheta}:= -\frac{1}{2(d-1)} B(X,\vartheta(Y)), \quad X,Y\in \mathfrak{g}
\label{cb.1}\end{equation}
to $\mathfrak{p}$ induces a $G$-invariant metric on $\bbH^d$ which is precisely the hyperbolic metric $g_{\hyp}$.  Suppose now that $(X,g_X)$ is a complete finite volume oriented hyperbolic manifold of dimension $d$ such that    
\begin{equation}
   X= \Gamma \setminus \bbH^d = \Gamma\setminus G/K
\label{cb.2}\end{equation}
for some discrete subgroup $\Gamma$ of $G$, so that the fundamental group $\pi_1(X)$ is naturally identified with $\Gamma$ and the metric $g_X$ on $X$ lifts to give the hyperbolic metric $g_{\hyp}$ on $\bbH^d$.  Let $\varrho: G\to \GL(V)$ be an irreducible representation on a complex vector space $V$ of complex dimension $k$.  Then the restriction of $\varrho$ to $\Gamma$ induces a unimodular representation of the fundamental group of $X$, which in turn induces a flat vector bundle 
\begin{equation}
   E:= \bbH^d\times_{\left.\varrho\right|_{\Gamma}} V \quad \mbox{on} \; X.
\label{cb.3}\end{equation}
We can instead restrict $\varrho$ to the maximal compact subgroup $K$ and consider the associated homogeneous vector bundle 
\begin{equation}
  \tE:= G\times_{\left.  \varrho\right|_K}V\quad \mbox{on} \quad \bbH^d= G/K
\label{cb.4}\end{equation}
and the corresponding locally homogeneous vector bundle $\Gamma\setminus \tE$ over $X=\Gamma\setminus \bbH^d$.  The space of smooth sections $\CI(\bbH^d;\tE)$ of $\tE$ is canonically identified with 
\begin{equation}
\CI(G;\varrho):= \left\{ f\in \CI(G;V)\; | \; f(gk)= \varrho(k^{-1})f(g)\; \forall g\in G, \ \forall k\in K  \right\}.
\label{cb.5}\end{equation}
Similarly, the space of smooth sections of $\CI(X; \Gamma\setminus \tE)$ is canonically isomorphic to 
\begin{equation}
  \CI(\Gamma\setminus G;\varrho):= \left\{ f\in \CI(G;\varrho) \; | \; f(\gamma g)= f(g) \; \forall g\in G, \ \forall \gamma\in \Gamma \right\}.
\label{cb.6}\end{equation}

From \cite[Proposition~3.1]{MM1963}, we know that there is a canonical vector bundle isomorphism 
\begin{equation}
      \Phi: \Gamma\setminus \tE \to E 
\label{cb.7}\end{equation}
explicitly given by 
\begin{equation}
\begin{array}{lccc}
  \Phi: &  \Gamma\setminus G\times_{\left.\varrho\right|_{K}} V & \to & G/K \times_{\left.  \varrho\right|_{\Gamma}}V  \\
          &  [g,v]_{\Gamma \setminus \tE} & \mapsto & [g, \varrho(g)v]_E,
\end{array}
\label{cb.8}\end{equation}
where $[g,v]_{\Gamma\setminus \tE}$ and $[g,v]_E$ denote the corresponding points in $\Gamma\setminus \tE$ and $E$ after taking the quotient by the actions of $\Gamma$ and $K$.  Using this natural isomorphism, we can equip $E$ with a canonical bundle metric.  More precisely, by \cite[Lemma~3.1]{MM1963}, there exists an inner product 
$\langle \cdot, \cdot\rangle$ on $V$ such that 
\begin{gather}
\label{cb.9a} \langle \varrho(Y)u,v\rangle= -\langle u, \varrho(Y) v\rangle\;  \forall Y\in \mathfrak{k}, \;  \forall u,v\in V; \\
\label{cb.9b} \langle \varrho(Y)u,v\rangle= \langle u, \varrho(Y) v\rangle\;  \forall Y\in \mathfrak{p}, \;  \forall u,v\in V.
\end{gather}   
  Clearly, $\left. \varrho\right|_K$ is unitary with respect to this inner product, which means that $\langle \cdot, \cdot \rangle$ induces a  bundle metric $h_{\Gamma\setminus \tE}$ on $\Gamma\setminus \tE$,
\begin{equation}
  h_{\Gamma\setminus \tE}\left( [g,v]_{\Gamma\setminus \tE}, [g,w]_{\Gamma\setminus \tE} \right):= \langle v,w\rangle,
\label{cb.10}\end{equation}
and hence a corresponding bundle metric $h_E$ on $E$ via the isomorphism \eqref{cb.7}.  This bundle metric $h_E$ and the flat connection on $E$ allow to define a de Rham operator and a Hodge Laplacian,
\begin{equation}
  \eth_E= d_E+ d_E^*, \quad \Delta_E= \eth_E^2= d_E d_E^* + d_E^* d_E,
\label{cb.11}\end{equation}
where $d_E: \Omega^*(X;E)\to \Omega^{*+1}(X;E)$ is the exterior derivative acting on differential forms taking values in $E$ and $d_E^*$ is its formal adjoint with respect to the $L^2$-inner product induced by $g_X$ and $h_E$.

Let $G=NAK$ be the Iwasawa decomposition of $G$ as in \cite[\S2]{MP2012} and let $M$ be the centralizer of $A$ in $K$.  Let $\mathfrak{n}, \mathfrak{a}$ and $\mathfrak{m}$ be the Lie algebras of $N$, $A$ and $M$ respectively.  Consider the group $P_0:= NAM$.  Recall that $\dim_{\bbR}\mathfrak{a}=1$.  Equip $\mathfrak{a}$ with the norm induced by the restriction of \eqref{cb.1}.  Let $H_1$ be the unique vector of norm $1$ such that the positive restricted root, implicit in the choice of $N$, is positive on $H_1$.  Then every $a\in A$ can be written as $a=\exp(\log a)$ for a unique $\log a\in \mathfrak{a}$, where $\exp: \mathfrak{a}\to A$ is the exponential map.  For $t\in \bbR$, set $a(t):= \exp(t H_1)$.  Given $g\in G$, we define $n(g)\in N$, $H(g)\in \bbR$ and $\kappa(g)\in K$ by
$$
      g= n(g)a(H(g))\kappa(g).
$$
If $P$ is a parabolic subgroup of $G$, then there is $k_P\in K$ such that $P=N_PA_PM_P$ with $N_P= k_P Nk_P^{-1}$, $A_P= k_P A k_P^{-1}$ and $M_P= k_P M k_P^{-1}$.  For instance, for $P=P_0$, we can take $k_{P_0}=1$.  For such a choice of $k_P$, let $a_P(t)= k_Pa(t) k_P^{-1}$.  For $g\in G$, define $n_P(g)\in N_P$, $H_P(g)\in \bbR$ and $\kappa_P(g)\in K$ by
$$
    g= n_P(g) a_P(H_P(g))\kappa_P(g)
$$
and consider the group isomorphism 
\begin{equation}
\begin{array}{lccc}
  R_P: & \bbR^+ & \to & A_P \\
               & t & \mapsto & a_P(\log t).
\end{array}
\label{dc.1}\end{equation}
We set $A^0_P[Y]:= R_P(Y,\infty)$ for $Y>0$.

\begin{definition}
A parabolic subgroup $P$ of $G$ is said to be \textbf{$\Gamma$-cuspidal} if $\Gamma\cap N_P$ is a lattice in $N_P$. 
Let $\mathfrak{P}_{\Gamma}$ be a set of representatives of the 
$\Gamma$-conjugacy classes of $\Gamma$-cuspidal parabolic subgroups of $G$.
\label{dc.2}\end{definition}  
Note that $\mathfrak{P}_{\Gamma}$ is a finite set with cardinality equal to 
the number of cusps of $X$.

\begin{assumption}
When $G=\SO_o(d,1)$, we will assume that for all $P\in\mathfrak{P}_{\Gamma}$ we
have
\begin{equation} 
  \Gamma \cap P= \Gamma\cap N_P,
\label{dc.3}\end{equation}
while when $G=\Spin(d,1)$, we will be more flexible and only assume that for 
all $P\in\mathfrak{P}_{\Gamma}$ we have
\begin{equation} 
  \pi(\Gamma) \cap \pi(P)= \pi(\Gamma)\cap \pi(N_P),
\label{dc.3c}\end{equation}
where $\pi: \Spin(d,1)\to \SO_o(d,1)$ is the canonical covering map.  Furthermore, when $G=\Spin(d,1)$ and \eqref{dc.3} does not hold for all $\Gamma$-cuspidal groups, we will also assume that 
\begin{equation}
     \varrho(e_{-1}) = \pm\Id,
\label{dc.3d}\end{equation}
where $e_{-1}\in \Spin(d,1)$ denotes the element different from the identity such that $\pi(e_{-1})$ gives the identity element in $\SO_0(d,1)$.  
\label{dc.3b}\end{assumption}
For a choice of $\mathfrak{P}_{\Gamma}$, there exists $Y_0>0$ such that for each $Y\ge Y_0$, there is a compact connected subset $C(Y)$ of $G$ and a decomposition 
\begin{equation}
 G= \Gamma\cdot C(Y) \sqcup \bigsqcup_{P\in \mathfrak{P}_{\Gamma}} \Gamma\cdot N_PA^0_P[Y]K
\label{dc.4}\end{equation}
such that for each $P\in \mathfrak{P}_{\Gamma}$, one has that 
\begin{equation}
  \left(\gamma\cdot N_P A^0_P[Y]K\right) \cap N_P A^0_P[Y]K \ne \emptyset \;  \Longleftrightarrow \; \gamma\in \Gamma\cap P.
\label{dc.5}\end{equation}
Set $\Gamma_P:= \Gamma\cap N_P= \Gamma\cap P$ when $G=\SO_o(d,1)$ and $\Gamma_P:=\pi(\Gamma)\cap\pi(N_P)=\pi(\Gamma)\cap \pi(P)$ when $G=\Spin(d,1)$. In the latter case, notice that $N_P$ is canonically identified with $\pi (N_P)$ via the canonical covering map, so that we will often denote $\pi(N_P)$ by $N_P$ to lighten the notation.  With this notation understood, if we set 
\begin{equation}
F_P(Y):= A^0_P[Y] \times \left( \Gamma_P\setminus N_P  \right) \cong (Y, \infty)\times \left( \Gamma_P\setminus N_P \right),
\label{dc.6}\end{equation}
then there is a corresponding decomposition of $X$, namely, there exists a compact manifold with smooth boundary $X(Y)$ such that 
\begin{equation}
X= X(Y) \sqcup \bigsqcup_{P\in \mathfrak{P}_{\Gamma}} F_P(Y)
\label{dc.7}\end{equation}
with $X(Y)\cap \overline{F_P(Y)}= \pa X(Y)= \pa F_P(Y)$ and  $\overline{F_{P_1}(Y)}\cap\overline{F_{P_2}(Y)}=\emptyset$ for $P_1\ne P_2$ in $\mathfrak{P}_{\Gamma}$.  
If $g_{P}$ is the invariant metric on $N_P$ induced by \eqref{cb.1} and $g_{T_P}$ is the corresponding metric on the quotient $T_P= \Gamma_P\setminus N_P$, then the restriction of the hyperbolic metric on $F_P(Y)$ is given by 
\begin{equation}
 \frac{dt^2 + g_{T_P}}{t^2}, \quad t\in (Y,\infty).
\label{dc.8}\end{equation} 
Since $N_P$ is abelian, notice that $g_{P}$ and $g_{T_P}$ are flat.  Since the hyperbolic metric is $G$-invariant, notice also that this description is consistent with the adjoint action of $A_P$ on the Lie algebra $\mathfrak{n}_P$ of $N_P$,
\begin{equation}
  \Ad(a_P(r))\eta= e^r \eta, \quad \eta\in \mathfrak{n}_P,  \quad  \Ad^*(a_P(r))\mu= e^{-r} \mu, \quad \mu\in \mathfrak{n}_P^*.
\label{dc.9}\end{equation}
By condition \eqref{cb.9b}  and the fact that $\mathfrak{a}\subset \mathfrak{p}$, there exists a basis $\{v_{P,1},\ldots, v_{P,k}  \}$ of $V$ orthonormal with respect to the admissible inner product $\langle\cdot, \cdot\rangle$ which is compatible with the weight decomposition of $V$ in terms of the action of $A_P$.  Thus, we suppose that 
\begin{equation}
\varrho(a_P(r)) v_{P,i}= e^{w_i r}v_{P,i} \quad \mbox{for some} \ w_i\in \bbR.  
\label{dc.10}\end{equation}
\begin{remark}
Since $A_P= k_PA k_P^{-1}$, we can assume that the weight $w_i$ does not depend on $P$.  
\end{remark}
If \eqref{dc.3} holds or $\varrho(e_{-1})=\Id$, let $\Xi_P: N_P\to \{1\}$ be the trivial group homomorphism, and otherwise let $\Xi_P: \pi(N_P)\to \bbS^1\subset \bbC^*\subset \End(V)$ be a choice of group homomorphism such that $\Xi_P(\pi(\gamma))\Id=\varrho(\kappa_P(\gamma))$ for all $\gamma\in \Gamma\cap P$.  In particular, by \eqref{dc.3d}, it is such that 
$\Xi(\Gamma_P)\subset \{1,-1\}$ and $\varrho(\gamma)= \varrho(n_P(\gamma))\Xi_P(\pi(\gamma))$ for all $\gamma\in \Gamma\cap P$.  Notice that using the decomposition $G=N_PA_P K$, the basis $\{v_{P,1},\ldots, v_{P,k}\}$ yields an orthonormal basis of sections 
$$
g\mapsto [n_P(g)a_P(H_P(g)), \Xi_P(\pi(n_P(g)))v_{P,i}]_{\Gamma\setminus \tE}
$$
 of $\Gamma\setminus \tE$ over the cusp $F_P(Y)$, and hence under the isomorphism \eqref{cb.7}, an orthonormal basis of sections 
\begin{equation}
   \nu_{P,i}(gK):= [n_P(g)a_P(H_P(g)), \Xi_P(\pi(n_P(g)))\varrho(n_P(g)a_P(H_P(g)))v_{P,i}]_E \quad \mbox{of} \ E \; \mbox{over} \; F_P(Y).
\label{dc.11}\end{equation}

Let $\varrho_P: N_P\to \End(V)$ be the restriction of $\varrho$ to $N_P$ twisted by $\Xi_P$, so that 
$$
         \varrho_P(n)= \Xi_P(\pi(n)) \varrho(n) \quad \forall  n\in N_P.
$$
The representation $\varrho_P$ defines by restriction to $\Gamma_P$ a flat vector bundle $E_P:= N_P\times_{\left. \varrho_P\right|_{\Gamma_P}} V$ on $\pa F_P(1)=\Gamma_P\setminus N_P=T_P$ and the basis $\{v_{P,i}\}$ induces a basis of sections  
\begin{equation}
   \nu_{N_P,i}(\Gamma_P n):= [n, \varrho_P(n) v_{P,i}]_{E_P}.
\label{dc.12}\end{equation}
The admissible product $\langle\cdot, \cdot \rangle$ naturally induces a bundle metric $h_{E_P}$ on $E_P$, namely the one obtained by declaring $\{\nu_{N_P,i}\}$ to be an orthonormal basis of sections.  By our choice of basis $\{v_{P,i}\}$, notice that on $F_P(Y)$, we have the following relation,
\begin{equation}
\nu_{P,i}(\Gamma n a_P(r))= e^{w_i r} \nu_{N_P,i}(\Gamma_P n) \quad \forall n\in N_P.
\label{dc.13}\end{equation}
Let 
\begin{equation}
    \eth_{E_P}= d_{E_P}+ d_{E_P}^*\quad \mbox{and} \quad \Delta_{E_P}= \eth_{E_P}^2
\label{dc.15}\end{equation}
be the corresponding de Rham operator and Hodge Laplacian, where $d_{E_P}: \Omega^*(T_P;E_P)\to \Omega^{*+1}(T_P;E_P)$ is the exterior derivative and $d^*_{E_P}$ is its formal adjoint with respect to the $L^2$-inner product induced by $g_{T_P}$ and $h_{E_P}$.  There is a natural inclusion 
\begin{equation}
\begin{array}{lccc}
\iota_P: & \Lambda^p\mathfrak{n}^*_P\otimes V & \hookrightarrow & \Omega^p(T_P;E_P) \\ 
             & \omega\otimes v & \mapsto & \hat{\omega}\otimes \hat{v},
\end{array}
\label{dc.16}\end{equation}
where $\hat{v}(\Gamma_P n):= [n,\varrho_P(n)v]_{E_P}$ and $\left. \hat{\omega}\right|_{\Gamma_P n}= \omega$ under the natural identification 
\begin{equation}
      \Lambda^p(T^*(T_P))= T_P\times \Lambda^p\mathfrak{n}^*_P.
\label{dc.17}\end{equation}
In fact , $\Lambda^*\mathfrak{n}^*_P\otimes V$ is a $V$-valued Lie algebra complex with differential  $d_P=d_{\mathfrak{n}_P}+ d_{\Xi_P}$ given by
\begin{gather}\label{dc.18}d
_{\mathfrak{n}_P} \Phi(T_1,\ldots, T_{q+1})= \sum_{i=1}^{q+1} (-1)^{i+1}\varrho(T_i)\Phi(T_1,\ldots, \hat{T}_i, \ldots, T_{q+1}),  \quad \Phi\in \Lambda^q \mathfrak{n}^*_P\otimes V, \\
d_{\Xi_P} \Phi(T_1,\ldots, T_{q+1})= \sum_{i=1}^{q+1} (-1)^{i+1}\Xi_P(T_i)\Phi(T_1,\ldots, \hat{T}_i, \ldots, T_{q+1}),  \quad \Phi\in \Lambda^q \mathfrak{n}^*_P\otimes V,
\label{dc.18b}\end{gather}
where the ``$\ \hat{}\ $" above a variable denotes omission.  
Of course, the map \eqref{dc.18} is also a differential inducing another Lie algebra complex structure on $\Lambda^*\mathfrak{n}^*_P\otimes V$, in fact the same whenever $\Xi_P$ is the trivial homomorphism.

Now, the inner product \eqref{cb.1} and the admissible inner product $\langle\cdot, \cdot\rangle$ on $V$ induce an inner product on $\Lambda^q\mathfrak{n}^*_P\otimes V$ for each $q$.  Let $d^*_{\mathfrak{n}_P}: \Lambda^*\mathfrak{n}_P^*\otimes V\to  \Lambda^{*-1}\mathfrak{n}_P^*\otimes V$ be the adjoint of $d_{\mathfrak{n}_P}$ with respect to this inner product.  Following Kostant \cite{Kostant}, we can consider the corresponding de Rham and Hodge operators 
\begin{equation}
K_P:= d_{\mathfrak{n}_P}+ d^*_{\mathfrak{n}_P}, \quad L_P:= K_P^2= d_{\mathfrak{n}_P}d^*_{\mathfrak{n}_P}+ d_{\mathfrak{n}_P}^*d_{\mathfrak{n}_P}
\label{dc.19}\end{equation}
and identify the Lie algebra cohomology $H^*(\mathfrak{n}_P;V)$ induced by the differential $d_{\mathfrak{n}_P}$ of \eqref{dc.18} with the kernel of $L_P$,
\begin{equation}
H^q(\mathfrak{n}_P;V)\cong \cH^q(\mathfrak{n}_P;V):= \left\{ \Phi\in \Lambda^q\mathfrak{n}_P^*\otimes V \; | \; L_P \Phi=0  \right\}.
\label{dc.20}\end{equation}
On $\Lambda^q\mathfrak{n}_P^*\otimes V$, there is also a natural action of $A_P$ an its Lie algebra $\mathfrak{a}_P$ induced by 
\begin{equation}
\Lambda^q\Ad^*\otimes \varrho (H_P)\Phi(T_1,\ldots,T_q)= \varrho(H_P)\Phi(T_1,\ldots,T_q)- \sum_{i=1}^q \Phi(T_1,\ldots, [H_P,T_i],\ldots, T_q)
\label{dc.21}\end{equation}
for $\Phi \in \Lambda^q\mathfrak{n}_P^*\otimes V$ and $H_P:= k_P H_1 k_P^{-1}\in \mathfrak{a}_P$.  
\begin{proposition}
The operators $d_{\mathfrak{n}_P}$ and $d_{\mathfrak{n}_P}^*$ are equivariant with respect to the actions of $A_P$ and $\mathfrak{a}_P$.
\label{dc.22}\end{proposition}
\begin{proof}
A direct computation shows that 
\begin{equation}
d_{\mathfrak{n}_P}\circ (\Lambda^*\Ad^*\otimes \varrho)(H_P)= (\Lambda^{*+1}\Ad^*\otimes \varrho)(H_P)\circ d_{\mathfrak{n}_P}.
\label{dc.23}\end{equation}
Now, by property \eqref{cb.9b} of the admissible product and \eqref{dc.9}, the operator $(\Lambda^q\Ad^*\otimes \varrho)(H_P)$ is self-adjoint, so taking the adjoint of \eqref{dc.23} gives
\begin{equation}
d^*_{\mathfrak{n}_P}\circ (\Lambda^*\Ad^*\otimes \varrho)(H_P)= (\Lambda^{*-1}\Ad^*\otimes \varrho)(H_P)\circ d^*_{\mathfrak{n}_P}.
\label{dc.24}\end{equation}
\end{proof}
\begin{corollary}
The operators $K_P$ and $L_P$ are equivariant with respect to the action of $A_P$ and $\mathfrak{a}_P$.
\label{dc.25}\end{corollary}
However, unless $\Xi_P$ is the trivial homomorphism, the operator $d_{\Xi_P}$ and its adjoint $d_{\Xi_P}^*$ are not equivariant with respect to the action of $A_P$ and $\mathfrak{a}_P$.  A direct computation using \eqref{dc.9} shows that we have instead 
\begin{equation}
 \begin{gathered}
            d_{\Xi_P}\circ (\Lambda^*\Ad^*\otimes \varrho)(H_P)-  (\Lambda^*\Ad^*\otimes \varrho)(H_P)\circ d_{\Xi_P}= d_{\Xi_P}, \\
            d_{\Xi_P}^*\circ (\Lambda^*\Ad^*\otimes \varrho)(H_P)-  (\Lambda^*\Ad^*\otimes \varrho)(H_P)\circ d^*_{\Xi_P}= -d^*_{\Xi_P}.
\end{gathered}
\label{dc.26}\end{equation}

\begin{lemma}[van Est's theorem]
If $\Xi_P$ is the trivial homomorphism, then the map \eqref{dc.16} is a map of complexes which induces an isomorphism in cohomology. If instead $\Xi_P$ is a non-trivial homomorphism, then $H^*(T_P;E_P)=\{0\}$.  
\label{dc.27}\end{lemma}
\begin{proof} 
This is one of the many manifestations of van Est's theorem \cite{vE}. First, fixing a basis $\{w^1_q,\ldots, w_q^{k_q}\}$ of $\Lambda^q\mathfrak{n}^*_P\otimes V$, we can write a general element $\lambda\in \Omega^q(T_P;E_P)$ as
\begin{equation}
   \lambda= \sum_j h^j w_q^j, \quad h^j\in \CI(T_P).
\label{vE.2}\end{equation}
We can in this way extend the definition of $d_{\mathfrak{n}_P}$ to a differential on all of  $\Omega^*(T_P;E_P)$ by
$$
    d_{\mathfrak{n}_P}\lambda= \sum_j h^j d_{\mathfrak{n}_P}w_q^j.
$$
We can also introduce  another differential $d_f$ on the complex $\Omega^*(T_P;E_P)$ defined by
$$
        d_f \lambda:= \sum_j (dh^j\wedge w^j_q+ h^jd_{\Xi_P}w^j_q).  
$$
In other words, the differential $d_f$ corresponds to the differential of the flat vector bundle  
\begin{equation}
N_P\times_{\left.\Xi_P\right|_{\Gamma_P}} V\cong (N_P\times_{\left.\Xi_P\right|_{\Gamma_P}} \bbC)^k.
\label{kf.1}\end{equation}
Then a simple computation shows that $d_{\mathfrak{n}_P}$ and $d_f$ anti-commute.  Moreover, in terms of these differentials, we have that
$$
     d_{E_P}= d_f+d_{\mathfrak{n}_P}.  
$$
Using the action \eqref{dc.21} as well as Proposition~\ref{dc.22} and \eqref{dc.26}, the complex of $d_{E_P}=d_f+d_{\mathfrak{n}_P}$ can be seen as a double complex with bigrading given by declaring an element $\iota_P(\omega\otimes v_{P,i})$ of bidegree  $(-w_i+q,w_i)$ whenever $\omega\in \Lambda^q\mathfrak{n}^*_P$, where we recall that the weight $w_i$ was introduced in equation \eqref{dc.10}.  If $\Xi_P$ is the trivial homomorphism, then  the first page of the associated spectral sequence is $E_1=\Lambda^*\mathfrak{n}_P\otimes V$ with differential $d_1=d_{\mathfrak{n}_P}$, so that the spectral sequence degenerates at the second page $E_2=H^*(\mathfrak{n}_P;V)$, yielding the result.  If instead $\Xi_P$ is a non-trivial homomorphism, then the complex of the differential $d_f$ is acyclic.  Indeed, by \eqref{kf.1}, it suffices to show that the flat line bundle $L:=N_P\times_{\left.\Xi_P\right|_{\Gamma_P}} \bbC$ on $T_P$ has trivial cohomology.  Now, if $\{\gamma_1,\ldots,\gamma_{2n}\}$ is a basis of $\Gamma_P$, let $L_i\to \bbS^1$ be the flat line bundle with holonomy given by $\Xi_P(\gamma_i)$. By the Künneth theorem, we have that 
$$
        H^*(T_P;L)\cong \bigotimes_{i=0}^{2n} H^*(\bbS^1;L_i).
$$ 
Since $\Xi_{P}$ is non-trivial, at least one of the $L_i$ must have non-trivial holonomy, that is, trivial cohomology, and therefore $H^*(T_P;L)$ must vanish as claimed.  Thus, coming back to the spectral sequence, this means in this case that it degenerates at the first page $E_1=\{0\}$, from which the second statement follows.         
\end{proof}

\section{The cusp surgery metric and and the cusp surgery bundle}\label{sm.0}

The hyperbolic manifold $(X,g_X)$ has a natural compactification by a manifold with boundary $\bX$ obtained from the decomposition \eqref{dc.7} by replacing $F_P(Y)\cong (Y,\infty)\times T_P$ by 
$\hat{F}_P(Y)\cong (Y,\infty]\times T_P$,
\begin{equation}
  \bX= X(Y) \bigsqcup_{P\in \mathfrak{B}_{\Gamma}} \hat{F}_P(Y), \quad \pa \bX= \bigsqcup_{P\in \mathfrak{P}_{\Gamma}} \{\infty\}\times T_P.
\label{sm.1}\end{equation}
On $\bX$, we can choose a boundary defining function $\bx$ such that for each $P\in \mathfrak{P}_{\Gamma}$, $\bx= t^{-1}$ on $\hat{F}_P(Y)= (Y,\infty]\times T_P$ with $t$ the coordinate on the first factor.   Hence, in terms of $\bx$, the hyperbolic metric on $F_P(Y)$ is given by
\begin{equation}
   \frac{d\bx^2}{\bx^2}+ \bx^2 g_{T_P}
\label{sm.2}\end{equation}
and $\{\bx^{-w_i}\nu_{N_P,i}\}$ is an orthonormal basis of sections of $E$.  Let 
\begin{equation}
M= \bX\bigcup_{\pa\bX} \bX
\label{sm.3}\end{equation}
be the double of $\bX$ obtained by gluing two copies of $\bX$ along their boundary.  Let us denote by $\bX_1$ and $\bX_2$ the copies of $\bX$ in $M$ intersecting on their boundary and let $\bx_1$ and $\bx_2$ be there corresponding boundary defining functions.  We can equip $M$ with an orientation by declaring that $\bX_1$ has the same orientation as $\bX$ and $\bX_2$ has the opposite orientation.  The closed manifold $M$ has a distinguished hypersurface $Z\subset M$ corresponding to the intersection of $\bX_1$ and $\bX_2$,
\begin{equation}
    Z := \bX_1\cap \bX_2= \pa \bX_1=\pa\bX_2\cong \pa \bX= \bigsqcup_{P\in \mathfrak{P}_{\Gamma}} T_P.
\label{sm.4}\end{equation}  
The hypersurface $Z$ has a tubular neighborhood $\nu_Z: (-Y^{-1}, Y^{-1})\times Z\hookrightarrow M$ defined by
$$
    \nu_Z(s,\tau)= (s^{-1},\tau)\in \hat{F}_P(Y)\subset \bX_1  \quad \mbox{for}\; \tau\in T_P,  s\ge 0
$$
  and by
$$
    \nu_Z(s,\tau)= (-s^{-1},\tau)\in \hat{F}_P(Y)\subset \bX_2  \quad \mbox{for}\; \tau\in T_P,  s\le 0.
$$  
Let $x\in \CI(M)$ be the function which restricts to $\bx_1$ on $\bX_1$ and to $-\bx_2$ on $\bX_2$, so that $\nu_Z^*x$ is just the projection $(-Y^{-1},Y^{-1})\times Z\to (-Y^{-1},Y^{-1})$ on the first factor.     
Taking $Y$ bigger if needed, consider then on $M$ a smooth family of metrics $g_{\epsilon}$ parametrized by $\epsilon>0$ such that $\nu_Z^*g_{\epsilon}$ is given by 
\begin{equation}
  \frac{dx^2}{x^2+\epsilon^2} + (x^2+\epsilon^2)g_{T_P} \quad \mbox{on} \; (-e^{-Y},e^{-Y})\times T_P \quad \mbox{for} \; P\in \mathfrak{P}_{\Gamma},
\label{sm.5}\end{equation}
and which away from $Z$ converges smoothly to the hyperbolic metric $g_{X_i}$ on each copy $X_i$ of $X$ inside $M$.  To see that such families of metrics exist, let $\chi\in\CI(M)$ be a function taking values in $[0,1]$, of compact support in the image of $\nu_Z$ and identically equal to $1$ in a neighborhood of $Z$.  Then we can take
\begin{equation}
     g_{\epsilon}= (1-\chi)g_0+ \chi\left( (\nu_{Z})_*\left( \frac{dx^2}{x^2+\epsilon^2} + (x^2+\epsilon^2)g_{T_P}\right)\right)
\label{sm.5b}\end{equation}
where $g_0$ is the hyperbolic metric on each copy of $X$ in $M$.   For such a family of metrics, it is useful to consider the single surgery space of Mazzeo and Melrose \cite{mame1}
\begin{equation}
    X_s:= [M\times [0,1]_{\epsilon}, Z\times \{0\}]
\label{sm.6}\end{equation}
obtained by blowing up $Z\times \{0\}$ inside $M\times [0,1]_{\epsilon}$ in the sense of Melrose \cite{MelroseAPS}.
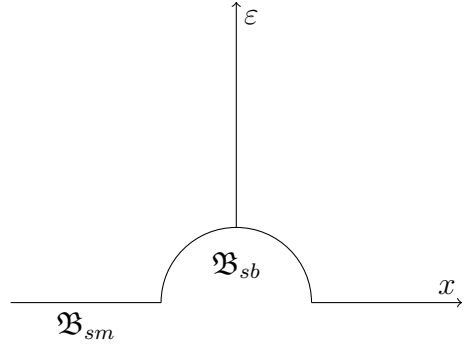
\begin{figure}[h]
\begin{tikzpicture}
surgery space
\draw(4,3) arc[radius=1, start angle=0, end angle=180];
\draw (0,3)--(2,3);
\draw[->] (4,3)--(6,3);
\draw[->] (3,4)--(3,7);

labeling
\node at (3,3.5) {$\bhs{sb}$};
\node at (1,2.7) {$\bhs{sm}$};
\node at (3.2,6.8) {$\epsilon$};
\node at (5.8,3.2) {$x$};

\end{tikzpicture}
\caption{The single surgery space $X_s$}
\label{fig.4}\end{figure}

  It is a manifold with corners with natural blow-down map $\beta_s: X_s\to M\times [0,1]_{\epsilon}$.  We denote by $\bhs{sb}:= \beta_s^{-1}(Z\times \{0\})$ the new boundary hypersurface introduced by the blow-up and by $\bhs{sm}:=\overline{\beta_s^{-1}(M\setminus Z)\times \{0\}}$ the lift of the old boundary hypersurface at $\epsilon=0$.

  There is also a boundary hypersurface at $\epsilon=1$, but it will not play any role in what follows.  Notice then that the function
\begin{equation}
     \rho:= \sqrt{x^2+\epsilon^2}
\label{sm.7}\end{equation}
is a boundary defining function for $\bhs{sb}$, so that $\frac{\epsilon}{\rho}$ is a boundary defining function for $\bhs{sm}$.

Let $E_i\to X_i$ be the flat vector bundle $E\to X$ on the copy $X_i$ of $X$ in $M$.  On $X_s$, we can then consider the vector bundle $E_s\to X_s$ which away from $\beta_s^{-1}(Z\times [0,1]_{\epsilon})$ is just the pull-back of $E_i\to X_i$ on $X_i\times [0,1]_\epsilon\subset X_s$, and near $\beta_s^{-1}(T_P\times [0,1]_{\epsilon})$, is spanned by the sections 
\begin{equation}
  \rho^{-w_i} \nu_{N_P,i}, \quad i\in\{1,\ldots,k\},
\label{sm.8}\end{equation}
which we declare to be linearly independent in each fiber of $E_s$ where they take values, as well as  smooth and bounded near $\bhs{sb}$.  Consider on $E_s$ a smooth bundle metric $h_s$ such that the sections \eqref{sm.8} form an orthonormal basis of sections of $E_s$ near $\beta_s^{-1}(T_P\times [0,1]_{\epsilon})$ for each $P\in \mathfrak{P}_{\Gamma}$ and such that away from $\bhs{sb}$, $h_s$ converges smoothly to the bundle metric $h_{E_i}$ on each copy $X_i$ of $X$ in $M$ as $\epsilon\searrow 0$.  As in \eqref{sm.5b}, such a bundle metric can be constructed using cut-off functions.  Notice that the flat connections on $E$ and $E_P$ for $P\in \mathfrak{P}_{\Gamma}$ induce a flat connection $d_{\epsilon}$ on $E_s$ on level sets of $\epsilon$ for $\epsilon>0$.  For instance, in terms of the local basis of sections \eqref{sm.8}, 
\begin{equation}
 d_{\epsilon}(\rho^{-w_i}\nu_{N_P,i})= -\frac{xw_i}{\rho}\left( \frac{dx}{\rho}\otimes \rho^{-w_i}\nu_{N_P,i}  \right)+ \rho^{-w_i}d_{E_P}\nu_{N_P,i},
\label{sm.9}\end{equation}
while away from $\beta_s^{-1}(Z\times [0,1]_{\epsilon})$ it converges smoothly to the flat connection of $E_i$ on each copy $X_i$ of $X$ inside $M$.  In particular, in terms of the usual cotangent bundle on $X_s$, this flat connection develops singularities at $\bhs{sb}$.  However, it is more useful to describe it in terms of the $\ed$-cotangent bundle $\Ed T^*X_s$ of \cite{ARS1},
which in $\nu_Z((-e^{-Y},e^{-Y})T_P)$ is spanned by  the sections 
\begin{equation}
         \frac{dx}{\rho}, \rho d\tau_1,\ldots, \rho d\tau_{2n}, 
\label{sm.9b}\end{equation}
considered as smooth, bounded and non-vanishing all the way to $\bhs{sb}$ as sections of $\Ed T^*X_s$, where the $\tau_i$ are a choice of coordinates on $T_P$.  
\begin{proposition}
The family of flat connections $d_{\epsilon}$ on $E_s$ induces an operator
$$
        d_{\epsilon}: \CI(X_s; \Lambda^*(\Ed T^*X_s)\otimes E_s)\to \rho^{-1} \CI(X_s; \Lambda^{*+1}(\Ed T^*X_s)\otimes E_s)
$$
which is an $\ed$-differential operator of order 1 in the sense of \cite{ARS1}, that is,
$$
d_{\epsilon}\in \Diff^1_{\ed}(X_s; \Lambda^*(\Ed T^* X_s)\otimes E_s)=\rho^{-1}\Diff^1_{\ephi}(X_s; \Lambda^*(\Ed T^* X_s)\otimes E_s),
$$
where $\phi:Z\to \mathfrak{P}_{\Gamma}$ is the projection which sends the connected component $T_P\subset Z$ onto $P\in \mathfrak{P}_{\Gamma}$. 
\label{sm.10}\end{proposition}
\begin{proof}
When the exterior derivative hits the `form part', this can be treated as in \cite[\S~2.2]{ARS1}.  When we differentiate sections, the first term on the right hand side of \eqref{sm.9} can also be treated as in \cite[\S~2.2]{ARS1}, so the only delicate point is the second term in \eqref{sm.9}.  However, since $\nu_{N_P,i}$ is in the image of \eqref{dc.16}, we see by Proposition~\ref{dc.22} and \eqref{dc.26}, that $d_{E_P}\nu_{N_P,i}$ has a part of weight $w_i$ (coming from $d_{\mathfrak{n}_P}$) and a part of weight $w_i-1$ (coming from $d_{\Xi_P}$) with respect to the action of $A_P$ and $\mathfrak{a}_P$, which means that the pointwise norm of $\rho^{-w_i+1}d_{E_P}\nu_{N_P,i}$ with respect to $g_{\epsilon}$ and $h_s$ is bounded and that  $\rho^{-w_i+1}d_{E_P}\nu_{N_P,i}$ is a smooth bounded section of $E_s$.  In fact, for the part of weight $w_i$, it is even better, namely $\rho^{-w_i}d_{\mathfrak{n}_P}\nu_{N_P,i}$ is a smooth bounded section of $E_s$.    
\end{proof}

Using the family of metrics $g_{\epsilon}$ and the bundle metric $h_s$, we can define an $L^2$-inner product for each $\epsilon>0$ and consider the formal adjoint $d^*_{\epsilon}$, as well as the corresponding de Rham operators and Hodge Laplacians,
\begin{equation}
  \eth_{\epsilon}= d_{\epsilon}+ d^*_{\epsilon}, \quad \Delta_{\epsilon}= \eth_{\epsilon}^2.
\label{sm.11}\end{equation}
We deduce the following from Proposition~\ref{sm.10}.
\begin{corollary}
The family of de Rham operator $\eth_\epsilon$ induces an operator
$$
        \eth_{\epsilon}: \CI(X_s; \Lambda^*(\Ed T^*X_s)\otimes E_s)\to \rho^{-1} \CI(X_s; \Lambda^{*+1}(\Ed T^*X_s)\otimes E_s).
$$
which is an $\ed$-differential operator of order 1.  Similarly, $\Delta_{\epsilon}$ induces an operator
$$
        \Delta_{\epsilon}: \CI(X_s; \Lambda^*(\Ed T^*X_s)\otimes E_s)\to \rho^{-2} \CI(X_s; \Lambda^{*+1}(\Ed T^*X_s)\otimes E_s)
$$
which is an $\ed$-differential operator of order 2, that is, 
$$
\Delta_{\epsilon}\in \Diff^2_{\ed}(X_s; \Lambda^*(\Ed T^* X_s)\otimes E_s)=\rho^{-2}\Diff^1_{\ephi}(X_s; \Lambda^*(\Ed T^* X_s)\otimes E_s).
$$

\label{sm.11b}\end{corollary}

To check that the uniform construction of the resolvent of \cite[Theorem~4.5]{ARS1} does apply to these operators, we need however to analyze more carefully the limiting behavior of $\eth_{\epsilon}$ as $\epsilon\searrow 0$.  First, to work with $b$-densities, we proceed as in \cite{ARS1} and consider instead the conjugated family of operators
\begin{equation}
   D_{\epsilon}:= \rho^n \eth_{\epsilon} \rho^{-n},
\label{sm.12}\end{equation}
where we recall that $\dim X= d= 2n+1$.  On $\CI(T_P;\Lambda^*(T^*T_P)\otimes E_P)$, let $W$ be the weight operator with respect to the action $\Lambda^*\Ad^* \otimes \varrho$ of $A_P$, so that 
\begin{equation}
     W(\omega\otimes \nu_{N_P,i})= (w_i-q)(\omega\otimes \nu_{N_P,i}) \quad \mbox{for}\; \omega\in \Omega^q(T_P).
\label{sm.12b}\end{equation}
Now, the section $\nu_{N_P,i}$ is not necessarily flat, but we see from Proposition~\ref{dc.22} and \eqref{dc.26} that 
$$
      W(d_{\mathfrak{n}_P}\nu_{N_P,i})=w_i d_{\mathfrak{n}_P}\nu_{N_P,i}, \quad       W(d_{\Xi_P}\nu_{N_P,i})=(w_i-1) d_{\Xi_P}\nu_{N_P,i}.
$$  
 Hence, in terms of the decomposition
\begin{equation}
\Lambda^q(\Ed T^*X_s)\otimes E_s= \rho^q\Lambda^q(T^*T_P)\otimes E_s\oplus \frac{dx}{\rho}\wedge \rho^{q-1}\Lambda^{q-1}(T^*T_P)\otimes E_s 
\label{sm.13}\end{equation}
in a tubular neighborhood of $T_P$ in $X_s$ and with respect to the basis of sections \eqref{sm.8} and the basis of forms \eqref{sm.9b}, one computes using \eqref{sm.9} that the operator $D_{\epsilon}$ takes the form
\begin{equation}
D_{\epsilon}= \left( \begin{array}{cc}  \frac{1}{\rho} \eth_{T_P} + K_P & -\rho\pa_x -(W+n)\frac{x}{\rho}  \\
\rho \pa_x-(W+n)\frac{x}{\rho}  & -\frac{1}{\rho} \eth_{T_P} -K_P   \end{array}  \right)
\label{sm.14}\end{equation}
near $\beta_s^{-1}(T_P\times [0,1])$, where $\eth_{T_P}$ is the de Rham operator on $(T_P,g_{T_P})$ acting on the vector bundle $E_P$ with flat connection obtained by declaring the sections $\nu_{N_P,i}$ to have differential 
\begin{equation}
     d \nu_{N_P,i}:= \frac{d\Xi_P}{\Xi_P}\otimes \nu_{N_P,i},
\label{sm.14b}\end{equation}
 while the de Rham operator of Kostant $K_P$ acts pointwise via the identification 
 $$
 \rho^{-W}\Omega^q(T_P;E_P)=\rho^{-W}\CI(T_P;\Lambda^q\mathfrak{n}^*_P\otimes V),
 $$ 
 keeping in mind that $K_P$ commutes with $W$ by Corollary~\ref{dc.25}.  Notice that $E_P$, equipped with the flat connection given by $\eqref{sm.14b}$, corresponds to the flat vector bundle 
 $$
 N_P\times_{\left.\Xi_P\right|_{\Gamma_P}} V
 $$  
 with holonomy representation given by the restriction of $\Xi_P$ to $\Gamma_P\cong \pi_1(T_P)$.  In particular, it is a trivial flat vector bundle when $\Xi_P$ is the trivial group homomorphism, but otherwise has no flat section and gives rise to an acyclic complex of differential forms.    

 The vertical operator of \cite[Definition~4.1]{ARS1} is then obtained by restricting the action of $\rho D_{\epsilon}$ to $\bhs{sb}$ in $X_s$.  Now, when we restrict the action of
$$
\rho D_{\epsilon}=\left( \begin{array}{cc}  \eth_{T_P}+ \rho K_P & -\rho^2\pa_x -(W+n)x  \\
\rho^2 \pa_x-(W+n)x  & - \eth_{T_P}- \rho K_P   \end{array}  \right)
$$ 
to the connected component  $\bhs{sb,P} := \beta_s^{-1}(T_P)$ of  $\bhs{sb}$, we get the constant family
\begin{equation}
 D_{v,P}= \left( \begin{array}{cc} \eth_{T_P}  & 0 \\ 0 & -\eth_{T_P} \end{array} \right).
\label{sm.15}\end{equation}
  More precisely, using the angle $\theta=\arctan \frac{x}{\epsilon}$ gives a natural identification $\bhs{sb,P}= [-\frac{\pi}2, \frac{\pi}2]\times T_P$, and $D_{v,P}$ is seen as a family of operators on $T_P$ parametrized by $\theta\in [-\frac{\pi}2,\frac{\pi}2]$.  In particular, by Lemma~\ref{dc.27}, the kernel of $D_{v,P}$ is trivial if the homomorphism $\Xi_P$ is non-trivial, and otherwise is a vector bundle over $[-\frac{\pi}2,\frac{\pi}2]$ with space of smooth sections canonically identified with 
$\CI([-\frac{\pi}2,\frac{\pi}2];\Lambda^*\mathfrak{n}_P^*\otimes V\oplus\Lambda^*\mathfrak{n}_P^*\otimes V)$ via the isomorphism
\begin{equation}
      \Psi_P: \CI([-\frac{\pi}2,\frac{\pi}2];\Lambda^*\mathfrak{n}_P^*\otimes V\oplus\Lambda^*\mathfrak{n}_P^*\otimes V)\to \CI([-\frac{\pi}2,\frac{\pi}2]; \ker D_v)
\label{sm.16}\end{equation}
defined by
$$
  \Psi_P((\omega,v_i),(\eta,v_j))= (\rho^{p-w_i}\omega\otimes \nu_{N_P,i}, \rho^{q-w_j} \frac{dx}{\rho}\wedge \eta \wedge \nu_{N_P,j})
$$
for $\omega\in\CI([-\frac{\pi}2,\frac{\pi}2]; \Lambda^p\mathfrak{n}_P^*)$ and $\eta\in\CI([-\frac{\pi}2,\frac{\pi}2]; \Lambda^q\mathfrak{n}_P^*)$.

Let $\Pi_h$ be the fiberwise $L^2$-orthogonal projection onto the kernel of $D_v$ with respect to the fiber bundle 
$$
    \bhs{sb,P} = [-\frac{\pi}2,\frac{\pi}2]\times T_P\to  [-\frac{\pi}2,\frac{\pi}2].
$$ 
Then the model operator $D_b$ of \cite[Definition~4.2]{ARS1} is obtained by looking at the action of $D_{\epsilon}$ on the sections of $\ker D_v$,
$$
     D_b u:= \Pi_h\left.(D_{\epsilon}\widetilde{u})\right|_{\bhs{sb}} \quad \mbox{for} \; u\in \CI(\bhs{sb}; \ker D_v),
 $$
 where $\widetilde{u} \in \CI(X_s;\Lambda^*(\Ed T^*X_s)\otimes E_s)$ is any smooth extension which restricts to give $u$ on $\bhs{sb}$.  Hence, using the identification \eqref{sm.16} and the variable $X= \frac{x}{\epsilon}=\tan \theta$ on the interior of $\bhs{sb,P}$, we see that on $\bhs{sb,P}$, the model operator $D_b$ is trivial when the homomorphism $\Xi_P$ is non-trivial, and otherwise is given by
\begin{equation}
D_{b,P}= \left(  \begin{array}{cc} K_P & -\ang X\pa_X- (W+n)\frac{X}{\ang X} \\
\ang X \pa X -(W+n)\frac{X}{\ang X} & -K_P\end{array}\right),
\label{sm.17}\end{equation}
where $K_P$ is the de Rham operator of Kostant given in \eqref{dc.19} and $\ang X= \sqrt{1+X^2}$.  Using the notation 
\begin{equation}
  D(a):= \ang X^{-a} \ang X \pa_X \ang X^{a}= \ang X \pa_X + \frac{aX}{\ang X} \quad \mbox{for} \ a\in \bbR,
\label{sm.18}\end{equation}
this can be rewritten as
\begin{equation}
  D_{b,P}=  \left(  \begin{array}{cc} K_P & -D(W+n) \\
 D(-W-n) & -K_P\end{array}\right).
 \label{sm.19}\end{equation}
By Corollary~\ref{dc.25}, the operator $K_P$ commutes with the weight operator $W$, so that 
\begin{equation}
  D_{b,P}^2=  \left(  \begin{array}{cc} K_P^2 -D(W+n)D(-W-n) & 0 \\
 0 & K_P^2 -D(-W-n)D(W+n)\end{array}\right).
 \label{sm.20}\end{equation}
 The behavior is clearly different whether the homomorphism $\Xi_P$ is trivial or not.  For this reason, we will denote by $\mathfrak{P}_{\Gamma}^{\varrho}$ the subset of $\mathfrak{P}_{\Gamma}$ consisting of parabolic subgroups $P$ such that $\Xi_P$ is the trivial homomorphism.   

\begin{lemma}
The operators $D_{b}$ and $D_b^2$ are Fredholm as $b$-operators for the $b$-density $\frac{dX}{\ang X}$ for $X\in \bbR$ provided $\varrho\circ \vartheta\ne \varrho$, where $\vartheta$ is the Cartan involution.  
\label{sm.21}\end{lemma}
\begin{proof}
By \cite{MelroseAPS}, it suffices to check that the indicial family $I(D_b^2,\lambda)$ is invertible for all $\lambda\in \bbR$.  Clearly, it suffices to check that $I(D_{b,P}^2,\lambda)$ is invertible for each $P\in \mathfrak{P}^{\varrho}_{\Gamma}$ and $\lambda\in\bbR$.  Now, as shown in \cite[(2.12)]{ARS2}, the indicial family of $D(a)$ is $I(D(a),\lambda)=\pm(a-i\lambda)$ at $X=\pm \infty$, so that the indicial family of $D_{b,P}$ is given by
\begin{equation}
  I(D^2_{b,P},\lambda)= \left(  \begin{array}{cc} K_P^2 +(W+n)^2+ \lambda^2 & 0 \\
 0 & K_P^2 +(W+n)^2+\lambda^2\end{array}\right)
 \label{sm.22}\end{equation}
 at both ends.  Clearly, this is invertible for $\lambda\in \bbR\setminus\{0\}$, while at $\lambda=0$, it is invertible provided $W+n\ne 0$ on the nullspace of $K_P$,
$$
     \ker K_P= \cH^*(\mathfrak{n}_P;V)\cong H^*(\mathfrak{n}_P;V).
$$
Recall that the highest weight $\Lambda(\varrho)$ of $\varrho$ is given by
\[
\Lambda(\varrho)=k_1(\varrho)e_1+\cdots+k_{n+1}(\varrho)e_{n+1},\quad 
k_1(\varrho)\ge k_2(\varrho)\ge\cdots\ge k_n(\varrho)\ge|k_{n+1}(\varrho)|,
\]
where $(k_1(\varrho),\dots,k_{n+1}(\varrho))$ belongs to $\bbZ[\frac{1}{2}]^{n+1}$ if 
$G=\Spin(d,1)$, and to $\bbZ^{n+1}$, if $G=\SO_0(d,1)$. For $q=0,\dots,n$ let
\[
\lambda_{\varrho,q}:=k_{q+1}(\varrho)+n-q,
\]
and for $q=n+1,\dots,2n,$ let
\[
\lambda_{\varrho,q}:=-\lambda_{\varrho,2n-q}.
\]
Furthermore, let
\[
\begin{cases}\lambda_{\varrho,n}^+:=\lambda_{\varrho,q},\;\;\lambda_{\varrho,q}^-=-
\lambda_{\varrho,q};\;\;\text{if}\;k_{n+1}(\varrho)\ge 0,\\
\lambda_{\varrho,n}^+:=-\lambda_{\varrho,q},\;\;\lambda_{\varrho,q}^-=
\lambda_{\varrho,q};\;\;\text{if}\;k_{n+1}(\varrho)< 0.
\end{cases}
\]
Now, by \cite[(6.4), (6.7)]{Pfaff2017}, \cf \cite[ChapterVI.3]{Borel-Wallach},  we have that for $q\neq n$
\begin{equation}
 W+n=\lambda_{\varrho,q} \quad \mbox{when acting on} \; \cH^q(\mathfrak{n}_P;V), 
\label{sm.23}\end{equation}
while for $q=n$, there is a decomposition
\begin{equation}
  \cH^n(\mathfrak{n}_P;V)= \cH^n_+(\mathfrak{n}_P;V) \oplus \cH_-^n(\mathfrak{n}_P;V)
\label{sm.24}\end{equation}
into  eigenspaces of $W+n$ in such a way that $W+n=\lambda_{\varrho,n}^{\pm}$ when acting on $\cH^n(\mathfrak{n}_P;V)_{\pm}$ with $\lambda_{\varrho,n}^+= -\lambda_{\varrho,n}^-$.  Furthermore, if $\varrho\circ \vartheta\ne \varrho$ where $\vartheta$ is the Cartan involution, we see, for instance from \cite[(2.4) and (2.7)]{Pfaff2017}, that 
\begin{equation}
\lambda_{\varrho,q}>0\quad  \mbox{for} \; 0\le q<n \quad  \mbox{and} \quad \lambda^+_{\varrho,n}>0. 
\label{sm.24b}\end{equation}
This implies that 
$W+n\ne 0$ on $\ker K_P$ and that $I(D_{b,P}^2,0)$ is invertible as desired.

\end{proof}
\begin{remark}
From \cite[(2.4) and (6.6)]{Pfaff2017}, we see that 
$$
\lambda_{\varrho,k}\ge \lambda_{\varrho,k+1}+1, \quad \mbox{for}  \quad 0\le k\le n-2,
\quad 
\mbox{and that} 
\quad 
        \lambda_{\varrho,n-1}\ge \lambda^+_{\varrho,n}+1.
$$
\label{sm.24c}\end{remark}

To make use of Lemma~\ref{sm.21}, we will therefore assume that 
\begin{equation}
  \varrho\circ \vartheta\ne \varrho,
\label{sm.25}\end{equation}
where $\vartheta$ is the Cartan involution.  The operator $D_b$ is then Fredholm and the uniform construction of the resolvent of \cite[Theorem~4.5]{ARS1} does apply.  
\begin{theorem}
Suppose that the representation $\varrho$ satisfies condition \eqref{sm.25}.  Then  the family of operator $D_{\epsilon}$ has finitely many \textbf{small} eigenvalues, that is, there are finitely many eigenvalues of $D_{\epsilon}$ tending to $0$ as $\epsilon\searrow 0$.  Furthermore, the projection $\Pi_{\sma}$ on the eigenspace of small eigenvalues is a polyhomogeneous operator of oder $-\infty$ in the surgery calculus of \cite{mame1}.  More precisely, $\Pi_{\sma}\in\Psi^{-\infty,\cK}_{b,s}(X_s; \Lambda^*(\Ed T^*X_s) \otimes E_s)$ for some index family $\cK$ with $\inf \cK\ge 0$.  Moreover, at $\epsilon=0$, $\Pi_{\sma}$ corresponds to the projection on $\ker_{L^2} D_b$ on $\bhs{sb}$ and to the projection on the $L^2$-kernel of $D_{E_i}=\bx_i^n \eth_{E_i}\bx_i^{-n}$ (using $b$-densities) on each connected component $\bX_i$ of $\bhs{sm}$.  In particular, the number of small eigenvalues counted with multiplicity is given by
\begin{equation}
  \rank \Pi_{\sma}= \dim \ker_{L^2} D_b+ 2\dim\ker_{L^2} \eth_E.
\label{sm.25c}\end{equation} 
\label{sm.25b}\end{theorem}
\begin{proof}
Assumption~1 of \cite[Theorem~4.5]{ARS1} is automatically satisfied by $D_{\epsilon}$, while Assumption 2 is a consequence of Lemma~\ref{sm.21}.  We can therefore apply \cite[Theorem~4.5]{ARS1} to conclude that there are finitely many eigenvalues, while the properties of the projections $\Pi_{\sma}$ follows from \cite[Theorem~4.5 and Corollary~5.2]{ARS1}.  For the formula for $\rank \Pi_{\sma}$, it is clear that
$$
\rank \Pi_{\sma}= \dim \ker_{L^2} D_b+ 2\dim \ker_{L^2} D_E,
$$
so it suffices to notice that the $L^2$-kernel of $D_{E}=\bx^n \eth_{E}\bx^{-n}$ in terms of the $b$-density $\frac{dg_X}{\bx^{2n}}$ is, via multiplication by $\bx^{-n}$, naturally isomorphic to the $L^2$-kernel of $\eth_{E}$ with respect to the density $dg_X$ of the hyperbolic metric.
\end{proof}

\section{Cusp degeneration of analytic torsion} \label{dat.0} 

To study the behavior of analytic torsion as $\epsilon\searrow 0$, we need to give a more precise description of the small eigenvalues of $D_{\epsilon}$.  For this, we must determine the $L^2$-kernel of $\eth_E$, $D_b$, and $D_{\epsilon}$ for $\epsilon>0$.  For $\eth_E$, it is a standard result.
\begin{proposition}
If $\varrho$ satisfies condition \eqref{sm.25} then $\ker_{L^2}\eth_E=\{0\}$.
\label{sm.25d}\end{proposition}
\begin{proof}
See for instance \cite[Proposition~8.1]{Pfaff2017}.
\end{proof}
 
 However, the $L^2$-kernel of $D_b$ is non-trivial in general.

\begin{proposition}
If \eqref{sm.25} holds, then the $L^2$-kernel of $D_b$ is given by
\begin{multline}
 \bigoplus_{P\in \mathfrak{P}^{\varrho}_{\Gamma}}\left( \left( \bigoplus_{q<n} \left( \begin{array}{c} 0 \\ \ang X^{-\lambda_{\varrho,q}}  \end{array} \right)\rho^{n-\lambda_{\varrho,q}}\cH^q(\mathfrak{n}_P;V) \right) \oplus  \left( \begin{array}{c} 0 \\ \ang X^{-\lambda_{\varrho,n}^+}  \end{array} \right)\rho^{n-\lambda^+_{\varrho,n}}\cH_+^n(\mathfrak{n}_P;V) \right.  \\
 \left. \oplus  \left( \begin{array}{c} \ang X^{\lambda_{\varrho,n}^-} \\ 0  \end{array} \right)\rho^{n-\lambda_{\varrho,n}^{-}}\cH_-^n(\mathfrak{n}_P;V) \oplus   \left( \bigoplus_{q>n} \left( \begin{array}{c}  \ang X^{\lambda_{\varrho,q}} \\ 0 \end{array} \right)\rho^{n-\lambda_{\varrho,q}}\cH^q(\mathfrak{n}_P;V) \right)\right).
\label{sm.26b}\end{multline}
\label{sm.26}\end{proposition}
\begin{proof}
Fix $P\in \mathfrak{P}^{\varrho}_{\Gamma}$.  Since $-D(-a)$ is the adjoint of $D(a)$, the operator \eqref{sm.20} is positive when restricted to a positive eigenspace of $K_P^2$.  Hence, the $L^2$-kernel of $D_{b,P}$ is the same as the operator
\begin{equation}
\left(  \begin{array}{cc} 0 & -D(W+n) \\
 D(-W-n) & 0\end{array}\right)
\label{sm.26c}\end{equation}
acting on sections on $\bbR$ taking values in $\ker K_P$.  Since the kernel of $D(a)$ is spanned by $\ang X^{-a}$, which is in $L^2$ with respect to the $b$-density $\frac{dX}{\ang X}$ if and only if $a>0$, the result follows from the description of the action of $(W+n)$ on $\cH^q(\mathfrak{n};V)$ given in \eqref{sm.23}, \eqref{sm.24} and \eqref{sm.24b} together with the fact that $\lambda_{\varrho,q}=-\lambda_{\varrho,2n-q}$ for $q<n$ and that $\lambda_{\varrho,n}^+=-\lambda_{\varrho,n}^-$.   
\end{proof}

To determine the kernel of $D_{\epsilon}$ for $\epsilon>0$, we need to give a description of the cohomology groups $H^q(X(Y);E)$.  In order to do this, we shall first compute the $L^2$-cohomology of $F_P(Y)$ with coefficients in $E$.  

\begin{lemma}
The $L^2$-cohomology $H^*_{(2)}(F_P(Y);E)$ of $(F_P(Y), \frac{dt^2+g_{T_P}}{t^2})$ with coefficients in $E$ is trivial if $\Xi_P$ is the trivial homomorphism, and otherwise is given by 
\begin{equation}
H^q_{(2)}(F_P(Y);E)= \left\{ \begin{array}{ll}H^q(T_P;E), & \mbox{if} \; q<n, \\
                                      \cH^n_+(\mathfrak{n}_P;V), & \mbox{if} \; q=n, \\
                                      \{0\}, &  \mbox{if} \; q>n.\end{array}  \right.
\label{zkf.2}\end{equation} 
\label{zkf.1}\end{lemma}
\begin{proof}
Unfortunately, we cannot use the $L^2$-Künneth formula of Zucker \cite{Zucker} as in the proof \cite[Proposition~2]{HHM}, the reason being that  in our setting, the function $\zeta$ of \cite[(2.3)]{Zucker} also depends on $w$ in \cite[(2.3)]{Zucker}.  We will instead use spectral sequences.

First, consider the space  $L^2\cA_{\phg}\Omega^q(\hat{F}_P(Y);E)$ of $L^2$-sections of $\Lambda^q({}^{\epsilon,d}T^*X_s)\otimes E_s$ on $\hat{F}_P(Y)\subset \bX\subset X_s$  at $\epsilon=0$ that have a polyhomogeneous expansion in $x$ in the sense of \cite{MelroseAPS}.  Then the subspaces
$$
L^2\cA_{\phg}'\Omega^q(\hat{F}_P(Y);E):= \{\nu\in L^2\cA_{\phg}\Omega^q(\hat{F}_P(Y);E) \; | \; d\nu\in L^2\cA_{\phg}\Omega^{q+1}(\hat{F}_P(Y);E)\}
$$
form a complex
\begin{equation}
\xymatrix{
  \cdots \ar[r]^-{d}  & L^2\cA_{\phg}'\Omega^q(\hat{F}_P(Y);E) \ar[r]^-{d} & L^2\cA_{\phg}'\Omega^{q+1}(\hat{F}_P(Y);E) \ar[r]^-{d} & \cdots
}
\label{zkf.4}\end{equation}
whose cohomology is precisely $H^*_{(2)}(F_P(Y);E)$.  Indeed, given a closed $L^2$-form with value in $E$ on $F_{P}(Y)$, it always admits an $L^2$-polyhomogneous representative, namely its harmonic part in the Hodge decomposition (the polyhomogeneity of $L^2$-harmonic forms can be seen for instance from \cite[Corollary~5.2]{ARS1} restricted to the face $\bhs{sm}$).  Moreover, if $\omega$ is a $L^2$-form such that $\eta=d\omega$ is $L^2$-polyhomogeneous, then replacing $\omega$ by its part in the image of $d^*$ in the Hodge decomposition, we can assume that $d^*\omega=0$ as well.  Hence, $(d+d^*)\omega=\eta$.  From \cite{HHM} and \cite{ARS1}, we thus see that $\omega$ has to be polyhomogeneous, showing that the cohomology of \eqref{zkf.4} is precisely $L^2$-cohomology.

Now, $d$ can be decomposed in three differentials,
$$
     d= d_t+ d_{\mathfrak{n}_P}+ d_f,
$$
where $d_{\mathfrak{n}_P}$ and $d_f$ are the differentials of the proof of Lemma~\ref{dc.27} acting on the second factor in $F_P(Y)=(Y,\infty)\times T_P$ and 
$$
       d_t\nu = dt\wedge \frac{\pa \nu}{\pa t}, 
$$
where $t\in (Y,\infty)$ is the coordinate on the first factor.  If we rewrite this as $d= d_A+d_B$ with $d_A= d_f$ and $d_B=d_{\mathfrak{n}_P}+d_t$, this can be seen as a double complex with bigrading given by declaring the elements $t^{w_i-q-\delta}\iota_P(\omega\otimes v_{P,i})$ and $t^{w_i-q-\delta}\frac{dt}t\wedge\iota_P(\omega\otimes v_{P,i})$ for $\delta>0$ 
respectively of bidegrees $(q-w_i, w_i)$ and $(q-w_i,w_i+1)$ whenever $\omega\in \Lambda^q\mathfrak{n}^*_P$.  A subtle point in the definition of double complex is that we need the projection onto a bidegree component to still be in the space.  This is the reason why we are using polyhomogeneous $L^2$-form, since then $d_t$ automatically preserves such a space, which ensures in turn that the projection of an element of $L^2\cA_{\phg}'\Omega^q(\hat{F}_P(Y);E)$ onto one of its bidegree components is again in that space.  

If $\Xi_P$ is a non-trivial homomorphism, then by Lemma~\ref{dc.27}, the corresponding spectral sequence degenerates at the first page with $E_1=\{0\}$, so the result follows.  If instead $\Xi_P$ is the trivial homomorphism, then the first page is 
$$
  E_1=L^2\cA_{\phg}'\Omega^*((Y,\infty);\Lambda^*(\mathfrak{n}^*_P)\otimes V)
$$
with differential $d_1= d_B= d_{\mathfrak{n}_P}+d_t$.  Moreover, the spectral sequence degenerates at the second page, which is just the cohomology of $(E_1,d_1)$.  To compute this cohomology, we can notice that the decomposition $d_1= d_{\mathfrak{n}_P}+d_t$ induces a structure of double complex with bigrading obtained by declaring $\omega\otimes v_{P,i}$ and $\frac{dt}t\otimes \omega\otimes v_{P,i}$ respectively of bidegrees $(q,0)$ and $(q,1)$ whenever $\omega\in \Lambda^q\mathfrak{n}^*_P$.  The corresponding spectral sequence has first page given by
$$
      E_1= L^2\cA_{\phg}'\Omega^*((Y,\infty); \cH^*(\mathfrak{n}_P;V))
$$
with differential $d_1= d_t$.  Hence, the spectral sequence degenerates at the second page $E_2$, which shows that the $L^2$-cohomology is identified with 
$E_2$, which is just the cohomology of $L^2\cA_{\phg}'\Omega^*((Y,\infty); \cH^*(\mathfrak{n}_P;V))$ with differential $d_t$.  

Now, let
$$
\WH(\alpha):= \{ f\in t^{\alpha} L^2((Y,\infty); \frac{dt}{t})\; |  \; df=0\}
$$
be the weighted $L^2$ cohomology of degree zero and weight $\alpha\in \bbR\setminus \{0\}$ on the interval $(Y,\infty)$ for the measure $\frac{dt}{t}$.  Recall from
\cite{HHM} that 
$$
    \WH(\alpha)=  \left\{  \begin{array}{ll} \bbC, & \mbox{if} \; \alpha>0, \\
               \{0\}, & \mbox{otherwise},  \end{array}\right.
$$
while the corresponding cohomology group in degree $1$ vanishes unless $\alpha=0$, in which case it is infinite dimensional.  Hence, 
the above discussion shows that $H^*_{(2)}(F_P(Y);E)$   is always trivial for $q=2n+1$, while using \eqref{sm.23} is given by
\begin{equation}
H^q_{(2)}(F_P(Y));E) = \WH(\lambda_{\varrho,q})\otimes H^q(T_P;E)\cong \left\{ \begin{array}{ll} H^q(T_P;E), & q<n, \\ \{0\}, & q>n,  \end{array}\right.
\label{l2.2}\end{equation}
for $q\notin\{n, 2n+1\}$ and by
\begin{equation}
\begin{aligned}
  H^n_{(2)}(F_P(Y);E)&= \WH(\lambda_{\varrho,q}^+)\otimes \cH^n_+(\mathfrak{n}_P;V) \oplus \WH(\lambda_{\varrho,q}^-)\otimes \cH^n_-(\mathfrak{n}_P;V) \\
                                   &= \cH^n_+(\mathfrak{n}_P;V)
\end{aligned}  
\label{l2.3}\end{equation} 
for $q=n$, giving the desired result when $\Xi_P$ is the trivial homomorphism.  

\end{proof}

\begin{proposition}
The cohomology group $H^q(X(Y);E)$ is trivial for $q<n$ and $q=2n+1$, while the inclusion $\iota_Y: \pa X(Y) \hookrightarrow X(Y)$ induces an isomorphism
\begin{equation}
  \iota_Y^*: H^k(X(Y);E) \to H^k(\pa X(Y);E) \quad \mbox{for} \; n<k\le 2n
\label{sm.27}\end{equation}
and an inclusion 
\begin{equation}
  \iota^*_Y: H^n(X(Y);E) \hookrightarrow H^n(\pa X(Y);E)  
  \label{sm.28}\end{equation}
such that $\dim H^n(X(Y);E)= \frac12 \dim H^n(\pa X(Y);E)$.
\label{l2.1}\end{proposition}
\begin{proof} 
The result follows from Lemma~\ref{zkf.1}, the fact implied by Proposition~\ref{sm.25d} that 
\begin{equation}
H^*_{(2)}(X;E)=\ker \eth_E=\{0\}
\label{l2.3b}\end{equation}
and the Mayer-Vietoris long exact sequence in $L^2$-cohomology
\begin{equation}
\xymatrix{
 \ar[r] & H^q_{(2)}(X;E) \ar[r] & H^q(X(Y);E)\oplus \left(\underset{{P\in\mathfrak{P}_{\Gamma}}}{\bigoplus}H^q_{(2)}(F_P(Y);E)\right) \ar[r] & H^q(\pa X(Y);E) \ar[r] & 
}  
\label{l2.4}\end{equation}
induced by the decomposition \eqref{dc.7} of $X$.
\end{proof}
\begin{remark}
When \eqref{dc.3} holds for each $P\in \mathfrak{P}_{\Gamma}$, Proposition~\ref{l2.1} is proved by Pfaff in \cite[\S~8]{Pfaff2017} using the approach of Harder \cite{Harder}.  See also \cite{MFP2012} for a proof when $n=1$ and $G=\Spin(3,1)\cong \SL(2,\bbC)$.
\end{remark}
\begin{proposition}
For $\epsilon>0$, consider the restriction $E_{\epsilon}= \left. E_s\right|_{\beta_s^{-1}(M\times \{\epsilon\})}$ of $E_s$ on $\beta^{-1}_s(M\times \{\epsilon\})\cong M$.  Then the cohomology $H^*(M; E_\epsilon)$ of the complex induced by the flat connection $d_{\epsilon}$ is such that
\begin{equation}
  \dim H^q(M;E_{\epsilon}) = \left\{  \begin{array}{ll} 0, & \mbox{if} \; q\in \{0,2n+1\}; \\
                                                                            \kappa^{\varrho}_{\Gamma}\dim H^{q-1}(\mathfrak{n};V), & \mbox{if} \; 1\le q <n;   \\
                                                                            \kappa^{\varrho}_{\Gamma}\left(\dim H^{n-1}(\mathfrak{n};V) +\frac{1}2\dim H^{n}(\mathfrak{n};V)\right),& \mbox{if} \; q=n; \\
                                                               \kappa^{\varrho}_{\Gamma}\left(\dim H^{n+1}(\mathfrak{n};V) +\frac{1}2\dim H^{n}(\mathfrak{n};V)\right),& \mbox{if} \; q=n+1; \\
                                                               \kappa^{\varrho}_{\Gamma}\dim H^{q}(\mathfrak{n};V), & \mbox{if} \;  n+1<q\le 2n;                                                               \end{array}\right.
\label{sm.29b}\end{equation}
where $\kappa^{\varrho}_{\Gamma}:= \# \mathfrak{P}^{\varrho}_{\Gamma}$.
\label{sm.29}\end{proposition}
\begin{proof}
Consider the Mayer-Vietoris long exact sequence in cohomology
\begin{equation}
\xymatrix{
     \cdots \ar[r]^-{\pa_{q-1}}  & H^q(M;E_{\epsilon})\ar[r]^-{i_q} & H^q(\bX;E_{\epsilon})\oplus H^q(\bX;E_{\epsilon}) \ar[r]^-{j_q} & H^q(Z;E_{\epsilon})\ar[r]^-{\pa_q} &\cdots
 }    
 \label{sm.30}\end{equation}
 coming from a decomposition $M=\cU\bigcup \cV$, where $\cU$ and $\cV$ are open sets of $M$ containing  $\overline{X}_1$ and $\bX_2$ in $M$ such that 
$$
    \cU \cap \cV = \nu_Z((-\delta,\delta)\times Z)
$$
is a tubular neighborhood of $Z\cong \pa\bX$ in $M$.  In particular, the map $j_q$ is defined by $j_q(\omega_1,\omega_2)= \iota_Y^*\omega_1-\iota_Y^*\omega_2$.
Now, by Lemma~\ref{dc.27}, we have that
\begin{equation}
   H^q(Z;E_{\epsilon})\cong H^q(\pa X(Y); E)\cong \bigoplus_{P\in \mathfrak{P}^{\varrho}_{\Gamma}} H^q(T_P;E_P)\cong 
   \bigoplus_{P\in \mathfrak{P}^{\varrho}_{\Gamma}} H^q(\mathfrak{n}_P;V)\cong (H^q(\mathfrak{n};V))^{\kappa^{\varrho}_{\Gamma}},
\label{sm.31}\end{equation}
where the last isomorphism follows from the fact $\mathfrak{n}_P= \Ad(k_P)\mathfrak{n}$.  Hence, the result follows by plugging this into the long exact sequence \eqref{sm.30}, as well as the description of $H^*(X(Y);E_{\epsilon})\cong H^*(X(Y);E)$ given above, using the fact that the map $\iota^*_Y: H^q(X(Y);E)\to H^q(\pa X(Y);E)$ is an inclusion for all $q$.  

\end{proof}

This yields the following useful fact about the small eigenvalues of $D_{\epsilon}$.
\begin{corollary}
If the representation $\varrho$ satisfies condition \eqref{sm.25}, then $D_{\epsilon}$ has no small eigenvalues that are positive and $\Pi_{\sma}$ is the projection on the kernel of $D_{\epsilon}$ for $\epsilon>0$ and the projection on the $L^2$-kernel of $D_b$ for $\epsilon=0$.
\label{sm.32}\end{corollary}
\begin{proof}
It follows from Proposition~\ref{sm.26} and Proposition~\ref{sm.29} that $$\dim \ker_{L^2} D_b= \dim H^*(M;E_{\epsilon})= \ker D_{\epsilon} \quad \mbox{for} \quad \epsilon>0.$$  Since $\ker_{L^2}\eth_E=0$ by Proposition~\ref{sm.25d}, we see from \eqref{sm.25c} that $\rank \Pi_{\sma}= \dim \ker D_{\epsilon}$ for $\epsilon>0$.  Since 
$\ker D_{\epsilon}$ is obviously included in the range of $\Pi_{\sma}$, the result follows.  
\end{proof}

The fact that there are no positive small eigenvalues as $\epsilon\searrow 0$ greatly simplifies the description of the limiting behavior of analytic torsion as $\epsilon\searrow 0$.  

\begin{theorem}
For $\epsilon>0$, let us denote by $h_{\epsilon}$ the bundle metric on $E_{\epsilon}$ induced by $h_s$.  Then as $\epsilon\searrow 0$, the logarithm of the analytic torsion of $(M,g_\epsilon, E_{\epsilon}, h_\epsilon )$ has a polyhomogeneous expansion and its finite part is given by 
\begin{equation}
\FP_{\epsilon= 0} \log T(M;E_{\epsilon}, g_{\epsilon},h_{\epsilon})= 2 \log T(X;E,g_X,h_E)+ \log T(D_b^2).
\end{equation}
\label{limit.1}\end{theorem}
\begin{proof}
This is a particular case of \cite[Corollary~11.3]{ARS1} except for the fact that the bundle metric is not even in the sense of \cite[Definition~7.6]{ARS1}.  This later condition was there to invoke \cite[Corollary~7.8]{ARS1}.  However, in the present setting, the same argument works to prove the analog of \cite[Corollary~7.8]{ARS1} if we substitute the number operator $N_{H/Y}$ occurring in the definition of the even and odd expansions of \cite[(7.33) and (7.34)]{ARS1} by the weight operator $W$ of \eqref{sm.12b}.   Indeed, once we choose normalized sections as in \eqref{sm.8}, the only effect on the model operator \eqref{sm.14} is to add terms of order zero which do not depend

on $\rho$\footnote{Since $\frac{x}{\rho}=\sin\theta$ in the coordinates of \cite{ARS1}.} and these extra terms do not affect the parity of the operator.  In particular, for the part given by $K_P$, this latter point is a consequence Corollary~\ref{dc.25}.       
 
\end{proof}
To compute the contribution of $T(D_b^2)$ coming from $T_P$ for $P\in \mathfrak{P}^{\varrho}_{\Gamma}$, notice from \eqref{sm.20} that 
$$
\log T(D_{b,P}^2)=A_P +B_P,
$$
where $A_P$ denotes the contributions coming from $\ker K_P\oplus \ker K_P$, while $B_P$ denotes the contribution coming from $\ker K_P^{\perp}\oplus \ker K_P^{\perp}$.  
\begin{lemma}
The number $A_P$ does not depend on $P\in \mathfrak{P}^{\varrho}_{\Gamma}$ and is given by 
\begin{equation}
\begin{aligned}
A_P &= \frac12 \sum_{q<n} (-1)^q\left[ \log c_{\lambda_{\varrho,q}} - (2q+1)\log \left( 2\lambda_{\varrho,q} \right)\right] \dim\cH^q(\mathfrak{n};V) \\
         & \quad + \frac12 \sum_{q>n} (-1)^q\left[ \log c_{-\lambda_{\varrho,q}} +(2q+1)\log \left( -2\lambda_{\varrho,q} \right)\right] \dim\cH^q(\mathfrak{n};V) \\  
         & \quad + \frac{(-1)^n\log c_{\lambda^+_{\varrho,q}}}{2}\dim \cH^n(\mathfrak{n};V). \end{aligned}
\label{sm.35}\end{equation}
where 
\begin{equation}
c_b:= \int_{\bbR} \ang X^{-2b} \frac{dX}{\ang X} = \frac{\Gamma(b)\Gamma(\frac12)}{\Gamma(b+\frac12)}, \quad \mbox{for} \; b>0.
\label{sm.33c}\end{equation}
Furthermore, if $n$ is odd, this formula simplifies to
\begin{equation}
A_P =  \frac{(-1)^n\log c_{\lambda^+_{\varrho,q}}}{2}\dim \cH^n(\mathfrak{n};V)+  \sum_{q<n} (-1)^q\left[ \log c_{\lambda_{\varrho,q}} + (2n-2q)\log \left( 2\lambda_{\varrho,q} \right)\right] \dim\cH^q(\mathfrak{n};V)  .
\label{sm.36}\end{equation}
\label{sm.33}\end{lemma}
\begin{proof}
When we restrict the action of $D_{b,P}^2$ to the kernel of $\ker K_P$, we obtain the operator
$$
\left(  \begin{array}{cc}  -D(W+n)D(-W-n) & 0 \\
 0 &  -D(-W-n)D(W+n)\end{array}\right).
$$
Setting $\Delta(a):= -D(-a)D(a)$, we know from \cite[equation (2.16)]{ARS2} that 
\begin{equation}
  \log\det \Delta(a)=  \log c_{|a|}- \sign(a) \log(2|a|) \quad \mbox{for}  \; a\ne 0.
\label{sm.33b}\end{equation}
Hence, using the decomposition of $\ker K_P$ into eigenspaces of $W+n$ given in \eqref{sm.23} and \eqref{sm.24},
we see that
\begin{equation}
\begin{aligned}
A_P &= -\frac12 \sum_{q\ne n } \left[ (-1)^q q \log \det \Delta(-\lambda_{\varrho,q}) + (-1)^{q+1}(q+1)\log \det \Delta(\lambda_{\varrho,q})\right]\dim \cH^q(\mathfrak{n}_P;V) \\
        & \quad-\frac12  \left[ (-1)^nn \log \det \Delta(-\lambda^+_{\varrho,n})+ (-1)^{n+1}(n+1)\log\det\Delta(\lambda_{\varrho,n}^+) \right] \dim \cH^n_+(\mathfrak{n}_P;V) \\
        & \quad -\frac12  \left[ (-1)^nn \log \det \Delta(-\lambda^-_{\varrho,n})+ (-1)^{n+1}(n+1)\log\det\Delta(\lambda_{\varrho,n}^-) \right] \dim \cH^n_-(\mathfrak{n}_P;V) \\
        &= -\frac12 \sum_{q<n}\left[ (-1)^qq\log\left( 2\lambda_{\varrho,q} c_{\lambda_{\varrho,q}} \right) + (-1)^{q+1}(q+1)\log\left(\frac{c_{\lambda_{\varrho,q}}}{2\lambda_{\varrho,q}}  \right)  \right] \dim\cH^q(\mathfrak{n}_P;V)\\
        &\quad -\frac12 \sum_{q>n}\left[ (-1)^qq\log\left( \frac{c_{-\lambda_{\varrho,q}}}{-2\lambda_{\varrho,q}}  \right) + (-1)^{q+1}(q+1)\log\left( -2\lambda_{\varrho,q}c_{-\lambda_{\varrho,q}}  \right)  \right] \dim\cH^q(\mathfrak{n}_P;V) \\
        & \quad -\frac12 \left[ (-1)^n n\log \left(2\lambda_{\varrho,n}^+c_{\lambda_{\varrho,n}^+} \right)+ (-1)^{n+1}(n+1)\log\left( \frac{c_{\lambda_{\varrho,n}^+}}{2\lambda_{\varrho,n}^+} \right)   \right]  \frac{\dim\cH^n(\mathfrak{n}_P;V)}2  \\
        & \quad -\frac12 \left[ (-1)^n n\log\left( \frac{c_{\lambda_{\varrho,n}^+}}{2\lambda_{\varrho,n}^+} \right)   + (-1)^{n+1}(n+1)\log \left(2\lambda_{\varrho,n}^+c_{\lambda_{\varrho,n}^+} \right)   \right]  \frac{\dim\cH^n(\mathfrak{n}_P;V)}2  \\
        &= \frac12 \sum_{q<n} (-1)^q\left[ \log c_{\lambda_{\varrho,q}} - (2q+1)\log \left( 2\lambda_{\varrho,q} \right)\right] \dim\cH^q(\mathfrak{n}_P;V) \\
         & \quad + \frac12 \sum_{q>n} (-1)^q\left[ \log c_{-\lambda_{\varrho,q}} +(2q+1)\log \left( -2\lambda_{\varrho,q} \right)\right] \dim\cH^q(\mathfrak{n}_P;V) \\  
         & \quad + \frac{(-1)^n\log c_{\lambda^+_{\varrho,n}}}{2}\dim \cH^n(\mathfrak{n}_P;V).            
         \end{aligned}
\label{sm.34}\end{equation}
Since $\mathfrak{n}_P= \Ad k_P (\mathfrak{n})$, we see that $\dim\cH^q(\mathfrak{n}_P;V)=\dim\cH^q(\mathfrak{n};V)$ and the result follows.  Furthermore, when $n$ is odd, we know from \cite[Theorem~8.3 p.68]{Onishchik} or \cite[\S~3.2.5]{GW} that the representation $\rho$ is self-dual, which implies that there is a canonical isomorphism $E_P^*\cong E_P$ which is an isomorphism of flat vector bundles and of Hermitian vector bundles.  Hence, we see by Poincaré duality that 
\begin{equation}
\dim\cH^{2n-q}(\mathfrak{n}_P;V)=\dim H^{2n-q}(T_P;E_P)= \dim H^{q}(T_P;E_P)\cong \dim \cH^q(\mathfrak{n}_P;V),
\label{sm.37}\end{equation}
from which \eqref{sm.36} follows.
\end{proof}

To compute $B_P$, we need the following result.

\begin{lemma}
For $a\in\bbR$ and $b>0$, we have that
$$
 \log\det\left( \Delta(a)+b^2 \right) -\log\det\left( \Delta(-a)+b^2 \right) = -2\log\left( \frac{a+ \sqrt{a^2+b^2}}{b} \right)
$$
\label{md.7c}\end{lemma}
\begin{proof}
By the proof of \cite[(2.14)]{ARS2}, we have that
$$
       {}^{R}\Tr(e^{-t\Delta(a)})- {}^{R}\Tr(e^{-t\Delta(-a)})= 2 \int_0^{a}\sqrt{\frac{t}{\pi}} e^{-tu^2}du.
$$
Hence, for $\Re s>-\frac12$, the difference of the regularized zeta function of $\Delta(a)+b^2$ and $\Delta(-a)+b^2$ is given by
\begin{equation}
\begin{aligned}
  \zeta(s) &= \frac{2}{\Gamma(s)}\int_{0}^\infty t^s e^{-tb^2} \left( \int_0^{a}\sqrt{\frac{t}{\pi}} e^{-tu^2}du\right)\frac{dt}{t} 
  = \frac{2}{\Gamma(s)}\int_0^{a}\left(\int_0^\infty t^s \sqrt{\frac{t}{\pi}} e^{-t(u^2+b^2)}\frac{dt}{t}\right) du \\
  &= \frac{2\Gamma(s+\frac12)}{\sqrt{\pi}\Gamma(s)}\int_0^{a} \frac{du}{(u^2+b^2)^{s+\frac12}} 
  = \frac{2\Gamma(\frac12)s}{\sqrt{\pi}} \int_0^{a} \frac{du}{(u^2+b^2)^{\frac12}} + \mathcal{O}(s^2) \\
  &= 2s\log\left( \frac{a+ \sqrt{a^2+b^2}}{b} \right) + \mathcal{O}(s^2)\end{aligned}
  \end{equation}  
as $s\to 0$.  Hence, we see that 
$$
   \zeta'(0)= 2\log\left( \frac{a+ \sqrt{a^2+b^2}}{b} \right),
$$
from which the result follows.

\end{proof}

We will also need the fact that the positive eigenvalues of $K_P^2$ come in pairs.  Indeed, if $v\in \Lambda^q\mathfrak{n}^*_P\otimes V$ is such that $d_{\mathfrak{n}_P}^* v=0$ and 
$K^2_Pv= b^2v$ with $b>0$, then $v':=d_{\mathfrak{n}_P} v\in \Lambda^{q+1}\mathfrak{n}^*_P\otimes V$ is such that
$$
      K^2_P v'= b^2 v', \quad  d^*_{\mathfrak{n}_P} v'= K_P^2 v=b^2v.
$$
Furthermore, by Proposition~\ref{dc.22}, we have that
$$
      (W+n)v= a v\; \Longrightarrow \;  (W+n)v'= a v'.
$$
Hence, for $q\in \bbN$, $a\in \bbR$ and $b>0$,  let $V_{q,a,b}$ be the subspace of elements $v$ in $\ker(d^*_{\mathfrak{n}_P})\cap \Lambda^q\mathfrak{n}_P^*\otimes V$ such that
$$
          (W+n)v=a v \quad \mbox{and} \quad  K^2_Pv= b^2v,
$$          
so that we have the decomposition
\begin{equation}
 (\ker K_P)^{\perp}= \bigoplus_{q,a,b} \left( V_{q,a,b} \oplus d_{\mathfrak{n}_P}V_{q,a,b} \right)
\label{sm.38}\end{equation}
\begin{remark}
Conjugating with $k_P\in K$, we see that $\dim V_{q,a,b}$ does not depend on $P\in \mathfrak{P}^{\varrho}_{\Gamma}$.
\end{remark}
\begin{lemma}
The term $B_P$ does not depend on $P$ and is given by 
$$
   B_P=\sum_{q,a,b} (-1)^q  \log\left( \frac{a+\sqrt{a^2+b^2}}{b} \right) \dim V_{q,a,b}.
$$
\label{sm.40}\end{lemma}
\begin{proof}
Set $d(a,b)= \log\det \left( \Delta(a)+ b^2 \right)$.  From \eqref{sm.20} and the decomposition \eqref{sm.38}, we see that
\begin{equation}
\begin{aligned}
B_P&= \sum_{q,a,b} -\frac{\dim V_{q,a,b}}2 \left[ ((-1)^q q + (-1)^{q+1}(q+1))d(-a,b)  \right.\\
 & \left.  \hspace{4cm}+ ((-1)^{q+1}(q+1)+ (-1)^{q+2}(q+2))d(a,b) \right] \\
 &= \sum_{q,a,b} \frac{(-1)^{q+1}\dim V_{q,a,b}}2\left[ d(a,b)-d(-a,b) \right],
\end{aligned}
\label{sm.41}\end{equation}
so that the result follows by applying Lemma~\ref{md.7c}.
\end{proof}

Thus, we deduce from Theorem~\ref{limit.1} that,
\begin{equation}
\FP_{\epsilon= 0} \log T(M;E, g_{\epsilon},h_{\epsilon})= 2 \log T(X;E,g,h)+ \sum_{P\in\mathfrak{P}^{\varrho}_{\Gamma}}\left( A_P+B_P \right),
\label{mc.2}\end{equation}
where $A_P$ and $B_P$ do not in fact depend on $P\in \mathfrak{P}^{\varrho}_{\Gamma}$ and are computed explicitly in Lemma~\ref{sm.33} and Lemma~\ref{sm.40}.

\section{Cusp degeneration of the Reidemeister torsion}\label{be.0}

To study the behavior of the Reidemeister torsion under the above cusp degeneration, we can use the Mayer-Vietoris long exact sequence \eqref{sm.30}.  For $q\ne n$,  let $\mu^q_Z$ be an orthonormal basis of 
$$
   H^q(Z;E_{\epsilon})\cong \bigoplus_{P\in \mathfrak{P}^{\varrho}_{\Gamma}} H^q(T_P;E_P)
$$
with respect to the metrics $g_{T_P}$ and bundle metrics $h_{E_P}$ for $P\in \mathfrak{P}^{\varrho}_{\Gamma}$.  For $q=n$, using the decomposition \eqref{sm.24} and Lemma~\ref{dc.27}, we have a decomposition 
\begin{equation}
     H^n(Z;E_{\epsilon})= H^n_+(Z;E_{\epsilon}) \oplus H^n_-(Z;E_{\epsilon}) 
 \label{be.1}\end{equation}
with 
$$
      H^n(Z;E_{\epsilon})_{\pm}= \bigoplus_{P\in \mathfrak{P}^{\varrho}_{\Gamma}} H^n(T_P;E_P)_{\pm},
$$
where $H^n(T_P;E_P)_{\pm}$ is the image of $\cH^n(\mathfrak{n}_P;V)_{\pm}$ under the map \eqref{dc.16}.  Let $\mu_{\pm}^n$ be an orthonormal basis of $H^n(Z;E_{\epsilon})_{\pm}$ with respect to the metrics $g_{T_P}$ and the bundle metrics $h_{E_P}$ for $P\in \mathfrak{P}^{\varrho}_{\Gamma}$.  
Notice that, combining \eqref{l2.3} and \eqref{l2.3b} with \eqref{l2.4}, we see that the map 
\begin{equation}
  \pr_-\circ \iota_Y^*:  H^n(X(Y);E)\to H_-^n(Z;E_{\epsilon})
\label{be.2}\end{equation}
is an isomorphism, where $pr_-: H^n(Z;E_{\epsilon})\to H_-^n(Z;E_{\epsilon})$ is the projection induced by the decomposition \eqref{be.1}.  We consider then on $H^n(Z;E_{\epsilon})$ the basis 
$$
      \mu^n_Z:= (\iota^*_Y\mu^n_X,\mu^n_+),
$$
where $\mu^n_X$ is the basis of $H^n(\bX;E_{\epsilon})\cong H^n(X(Y);E)$ chosen such that $\pr_-\circ \iota^*_Y( \mu^n_X)= \mu^n_-$.  
\begin{remark}
The basis $\mu^n_{Z}$ is typically not orthonormal with respect to $g_{T_P}$ and $h_{E_P}$, but the change of basis from $\mu^n_Z$ to $(\mu^n_-,\mu^n_+)$ has determinant $1$.
\label{be.3}\end{remark}
One advantage of the basis $\mu_{Z}^q$ is that it is compatible with decomposition
$$
  H^q(Z;E_{\epsilon})= \ker \pa_q \oplus \Im \pa_q
$$
induced by the long exact sequence \eqref{sm.30}.
Now, the map  $j_q$ is explicitly given by
$$
      j_q(u,v)= \iota_Y^*u-\iota_Y^*v,
$$
so 
$$
   \ker \pa_q= \Im j_q = \Im   \iota_Y^*.
$$
Hence, on $H^q(X(Y);E)$, we have an induced basis $\mu_{X}^q$ such that $\iota_Y^*(\mu_{X}^q)=\left. \mu_{Z}^q\right|_{\ker\pa_q}$ (which gives back $\mu^n_X$ above when $q=n$).  This basis in turn induces a basis $\mu_{X}^q\oplus \mu^q_{X}$ on $H^q(\bX;E_{\epsilon})\oplus H^q(\bX;E_{\epsilon})$.  
Similarly, the map $i_q$ is given by
$$
    i_q(u) = (\iota_1^*u,\iota_2^*u),
$$
where $\iota_i: \overline{X}\hookrightarrow M$ for $i=1,2$ is the inclusion of each copy of $\bX$ in $M$.  In particular, we see that $\Im i_q \cong \Im(\iota_Y^*\iota_1^*)$ since $\iota_Y^*$ is injective.  We have in particular a natural decomposition
$$
    H^q(M;E_{\epsilon})= \Im \pa_{q-1} \oplus  \Im (\iota_Y^*\iota_1^*),
$$  
and the basis $\mu_{Z}^*$ induces a basis $\mu_{M}^q$ of $H^q(M;E_{\epsilon})$ compatible with this decomposition given by
$$
         \mu_{M}^q= (\pa_{q-1}(\left. \mu^{q-1}_{Z}\right|_{H^{q-1}(Z;E_{\epsilon})/\ker \pa_{q-1}}), (\iota_Y^*\iota_1^*)^{-1}(\left. \mu_{Z}^q\right|_{\ker\pa_q}). 
$$ 
With these choices of bases, we see that $|\det((\pa_q)_{\perp})|:= |\det(\pa_q: (\ker\pa_q)^{\perp}\to \Im \pa_q)|=1$.  For the map $i_q$ and $j_q$, there are some factors of $\sqrt{2}$ to take into account.  Indeed, if $\mu^q_{X}=\{u_1,\ldots u_{\ell}\}$ so that 
$ \mu_{X}^q\oplus \mu^q_{X}=\{(u_1,0),\ldots,(u_{\ell},0),(0,u_1),\ldots,(0,u_{\ell})\}$, then one can consider instead the orthonormal change of basis
$$
   (\frac{u_1}{\sqrt2},\frac{u_1}{\sqrt2}),\ldots, (\frac{u_{\ell}}{\sqrt2},\frac{u_{\ell}}{\sqrt2}), (\frac{u_1}{\sqrt2},-\frac{u_1}{\sqrt2}),\ldots, (\frac{u_{\ell}}{\sqrt2},-\frac{u_{\ell}}{\sqrt2})
$$  
compatible with the decomposition $H^q(\bX;E_{\epsilon})\oplus H^q(\bX;E_{\epsilon})= \ker j_q \oplus \Im j_q$ since 
$$
    \ker j_q= \spane \{ (\frac{u_1}{\sqrt2},\frac{u_1}{\sqrt2}),\ldots, (\frac{u_{\ell}}{\sqrt2},\frac{u_{\ell}}{\sqrt2})\}. 
$$ 
Since $j_q(\frac{u_j}{\sqrt2},-\frac{u_j}{\sqrt2})= \sqrt2 u_j$ and $i_q((\iota_1^*)^{-1}(u_j))= \sqrt2 (\frac{u_j}{\sqrt2},\frac{u_j}{\sqrt2})$, we thus see that
\begin{equation}
   |\det((j_q)_{\perp})|=|\det((i_q)_{\perp})|= 2^{\frac{\dim H^q(\bX;E_{\epsilon})}2}.
\label{rt.1}\end{equation}

\begin{theorem}
With the above choices of bases in cohomology, we have that the Reidemeister torsions of $M$ and $\bX\cong X(Y)$ are related by
$$
    \tau(M,E_{\epsilon},\mu_{M})= \frac{\tau(X(Y),E,\mu_{X})^2}{\tau(Z,E,\mu_{Z})}.
$$
Furthermore, if $n$ is odd, then $\tau(Z,E,\mu_{Z})=1$ and the formula simplifies to
$$
\tau(M,E_{\epsilon},\mu_{M})= \tau(X(Y),E,\mu_{X})^2.
$$
\label{rt.2}\end{theorem}
\begin{proof}
By the formula of Milnor, we have that
$$
           \frac{\tau(X(Y),E,\mu_{X})^2}{\tau(M,E_{\epsilon},\mu_M)\tau(Z,E,\mu_{Z})}= \tau(\cH),
$$
where $\tau(\cH)$ is the torsion of the complex \eqref{sm.30} with preferred basis given by $\mu_M$, $\mu_{X}\oplus \mu_{X}$ and $\mu_{Z}$.  Using \eqref{rt.1}, one computes that in fact
$$
 \tau(\cH)= \prod_{q} \left( |\det((j_q)_{\perp})|^{(-1)^{q+1}}|\det((i_q)_{\perp})|^{(-1)^{q}} \right)=1,
$$
from which the first formula follows.  For the second formula, 
notice first that by Remark~\ref{be.3}, we can replace the basis $\mu_Z$ with an orthonormal basis without changing the torsion.  Now, when $n$ is odd, we know from \cite[\S~3.2.5]{GW} that the representation $\varrho$ is self-dual.   This means that there is a canonical isomorphism $E^*\cong E$ which is an isomorphism of flat vector bundles and of Hermitian vector bundles.  Thus, by Poincar\'e duality, Milnor duality and \cite[Proposition~1.12]{Muller1993}, we have that
$$
    \tau(Z,E,\mu_{Z})^2=1\; \Longrightarrow \; \tau(Z,E,\mu_{Z})=1.
$$

\end{proof}

\section{Formula relating Analytic and Reidemeister torsions}  \label{cm.0}

By the Cheeger-M\"uller theorem of \cite{Muller1993}, we know that for $\varepsilon>0$, 
\begin{equation}
   \log \tau(M,E_{\epsilon},\mu_M)= \log T(M,E,g_{\varepsilon},h_{\varepsilon})- \log\left(  \prod_q [\mu^q_M| \omega^q]^{(-1)^q}\right),
\label{cm.1}\end{equation}
where $\omega^q$ is an orthonormal basis of harmonic forms with respect to $g_{\varepsilon}$ and $h_{\varepsilon}$ and $[\mu^q_M|\omega^q]= |\det W^q|$ with $W^q$ the matrix describing the change of basis 
$$
       (\mu_M^q)_i= \sum_j W^q_{ij} \omega_j^q.
$$
To obtain a formula relating analytic torsion and Reidemeister torsion on the hyperbolic manifold $(X,g)$, we will take the finite part at $\epsilon=0$ of the right hand side of \eqref{cm.1}.  By formula \eqref{mc.2}, we know how to compute the finite part of $\log T(M,E,g_{\varepsilon},h_{\varepsilon})$ as $\varepsilon\searrow 0$.  To compute the finite part of $\log\left(  \prod_q [\mu^q_M| \omega^q]^{(-1)^q}\right)$, we will proceed as in \cite[\S~3.3]{ARS2}.

First, by the definition of $c_b$ given in \eqref{sm.33c} and using \eqref{sm.23}, we see that an orthonormal basis of $\ker_{L^2}D_b$ is given by
\begin{equation}
\begin{gathered}
\frac{1}{\sqrt{c_{\lambda_{\varrho,q-1}}}}\begin{pmatrix}0 \\ \langle X\rangle^{-\lambda_{\varrho,q-1}}\end{pmatrix}\rho^{n-\lambda_{\varrho,q-1}}\mu^{q-1}_Z, \quad \mbox{in degree} \quad 1\le q<n;  \\
\left(\frac{1}{\sqrt{c_{\lambda_{\varrho,n-1}}}}\begin{pmatrix}0 \\ \langle X\rangle^{-\lambda_{\varrho,n-1}}\end{pmatrix}\rho^{n-\lambda_{\varrho,n-1}}\mu^{n-1}_Z,
\frac{1}{\sqrt{c_{\lambda_{\varrho,n}^+}}}\begin{pmatrix} \langle X\rangle^{-\lambda_{\varrho,n}^+} \\ 0\end{pmatrix}\rho^{n+\lambda_{\varrho,n}^+}\mu^{n}_-\right), \quad \mbox{in degree} \quad q=n; \\
\left( \frac{1}{\sqrt{c_{\lambda_{\varrho,n}^+}}}\begin{pmatrix}0 \\ \langle X\rangle^{-\lambda_{\varrho,n}^+}\end{pmatrix}\rho^{n-\lambda^+_{\varrho,n}}\mu^{n}_+,
\frac{1}{\sqrt{c_{-\lambda_{\varrho,n+1}}}}\begin{pmatrix} \langle X\rangle^{\lambda_{\varrho,n+1}} \\ 0\end{pmatrix}\rho^{n-\lambda_{\varrho,n+1}}\mu^{n+1}_Z \right) \quad \mbox{in degree} \quad q=n+1; \\
\frac{1}{\sqrt{c_{-\lambda_{\varrho,q}}}}\begin{pmatrix} \langle X\rangle^{\lambda_{\varrho,q}} \\ 0\end{pmatrix}\rho^{n-\lambda_{\varrho,q}}\mu^{q}_Z, \quad \mbox{in degree} \quad n+1\le q\le 2n.
\end{gathered}
\label{md.9}\end{equation}

Extending this basis smoothly to $(X_s,E_s)$ and then applying $\Pi_{\ker D_{\epsilon}}=\Pi_{\sma}$ to them gives, for $\epsilon$ small enough, a basis of $\ker D_{\epsilon}$ made of polyhomogeneous sections of $\Lambda^*\Ed T^*X_s\otimes E_s$.  Applying the Gram-Schmidt process, we can assume furthermore that this basis is orthonormal.  Hence, we see that there is an orthonormal  basis $\mu^{q}_{X_s}$ of $\ker D_{\epsilon}$ in degree $q$ of polyhomogeneous sections which restricts on $\bhs{sb}$ to give the basis \eqref{md.9} in degree $q$.   
 
 Now for $1\le q\le n$ and $v^{q-1}\in \mu^{q-1}_Z$, consider the family of  harmonic forms $\omega^q_{\epsilon}$ of degree $q$ with respect to $\eth_{\epsilon}$ such that in terms of the operator $D_{\epsilon}=\rho^n\eth_{\epsilon}\rho^{-n}$, we have that $\rho^n \omega^q_{\epsilon}\in \mu^q_{X_s}$ restricts to 
 \begin{equation}
 \frac{1}{\sqrt{c_{\lambda_{\varrho,q-1}}}}\begin{pmatrix}0 \\ \langle X\rangle^{-\lambda_{\varrho,q-1}}\end{pmatrix}\rho^{n-\lambda_{\varrho,q-1}}v^{q-1}= \frac{\rho^n \epsilon^{-\lambda_{\varrho,q-1}}}{\sqrt{c_{\lambda_{\varrho,q-1}}}}\ang X^{-2\lambda_{\varrho,q-1}}\frac{dX}{\ang X} \wedge v^{q-1} \quad \mbox{at} \quad \epsilon=0.
 \label{dual.1}\end{equation}
If $\epsilon^\mu w$ for $\mu>0$ is a higher order term in the expansion of $\omega^q_{\epsilon}$ as $\epsilon\searrow 0$, then since $D_{\epsilon}$ commutes with $\epsilon$, we see that $\rho^n w\in \ker D_{\epsilon}$, so in particular the restriction $\rho^n w_{sb,P}$ of $\rho^n w$ to $\bhs{sb,P}$ is in $\ker D_{b,P}$.  As opposed to the $L^2$-kernel of $D_{b,P}$, the full kernel of $D_{b,P}$ is more complicated and involves more substantially the operator $K_P$.  For us, what is important is that it decomposes into two parts,
$$
     \ker D_{b,P}= \ker_1 D_{b,P} \oplus \ker_2 D_{b,P},
$$ 
where $\ker_1 D_{b,P}$ and $\ker_2 D_{b,P}$ correspond to the elements of $\ker D_{b,P}$ which are sections on $\bbR$ taking values in $\ker K_P$ and $(\ker K_P)^{\perp}$ respectively.  Thus, the first factor contains the $L^2$-kernel and corresponds to the kernel of the operator in \eqref{sm.26c} for sections on $\bbR$ taking values in $\ker K_P$.  This part can be written explicitly using the fact that the kernel of $D(a)$ is spanned by $\ang{X}^{-a}$.  The second part is more intricate to describe, but the key observation for us will be that, since it comes from $(\ker K_P)^{\perp}$, it leads to no cohomological contributions as we will see.

Thus, if $q<n$, the above discussion implies  that 
\begin{equation}
\rho^n w_{sb,P}= \left( \begin{array}{c}  0 \\ \ang X^{-\lambda_{\varrho,q-1}}  \end{array} \right)\rho^{n-\lambda_{\varrho,q-1}} w_{q-1,P} + \left( \begin{array}{c}  \ang X^{\lambda_{\varrho,q}} \\0  \end{array} \right)\rho^{n-\lambda_{\varrho,q}} w_{q,P}+a_{q,P},
\label{dual.2}\end{equation} 
with  $w_{q,P}\in \cH^q(T_P;E_P)$, where the first two terms are in $\ker_1 D_{b,P}$ and $a_{q,P}\in \ker_2 D_{b,P}$.  
 This means  that 
 \begin{equation}
 \epsilon^{\mu} w_{sb,P}= \epsilon^{\mu-\lambda_{\varrho,q-1}}\ang X^{-2\lambda_{\varrho,q-1}}\frac{dX}{\ang X}\wedge w_{q-1,P}+ 
 \epsilon^{\mu-\lambda_{\varrho,q}}w_{q,P} + \epsilon^{\mu}\rho^{-n}a_{q,P}.
 \label{dual.3}\end{equation}
 Moreover, by Proposition~\ref{sm.29} and its proof, only the first term contributes cohomologically.  
 If instead $q=n$, we must replace \eqref{dual.2} by
 \begin{multline}
\rho^n w_{sb,P}= \left( \begin{array}{c}  0 \\ \ang X^{-\lambda_{\varrho,n-1}}  \end{array} \right)\rho^{n-\lambda_{\varrho,n-1}} w_{q-1,P}  \\ + \left( \begin{array}{c}  \ang X^{\lambda_{\varrho,n}^+} \\0  \end{array} \right)\rho^{n-\lambda_{\varrho,n}^+} w_{n,P}^+ + \left( \begin{array}{c}  \ang X^{-\lambda_{\varrho,n}^+} \\0  \end{array} \right)\rho^{n+\lambda_{\varrho,n}^+} w_{n,P}^- + a_{n,P}
\label{dual.4}\end{multline}  
with $w^{\pm}_{n,P}\in \cH^n(T_P;E_P)_{\pm}$ and $a_{n,P}\in \ker_2 D_{b,P}$.  Hence, in this case,
\begin{equation}
 \epsilon^{\mu} w_{sb,P}= \epsilon^{\mu-\lambda_{\varrho,n-1}}\ang X^{-2\lambda_{\varrho,n-1}}\frac{dX}{\ang X}\wedge w_{q-1,P}+ 
 \epsilon^{\mu-\lambda_{\varrho,n}^+}w_{n,P}^+  +  \epsilon^{\mu+\lambda^+_{\varrho,n}}w^-_{n,P}+ \epsilon^{\mu}\rho^{-n}a_{n,P}.
 \label{dual.5}\end{equation}
 Clearly the last term does not contribute cohomologically.  Hence, since $\mu>0$ and  $\lambda_{\varrho,n-1}>\lambda_{\varrho,n}^+ >0$ by Remark~\ref{sm.24c}, we see that \eqref{dual.3} and \eqref{dual.5} are cohomologically negligible compared to \eqref{dual.1} multiplied by $\rho^{-n}$ as $\epsilon \searrow 0$.  
Hence, asymptotically, as $\varepsilon\searrow 0$, we have that
$$
\omega^q_{\epsilon}\sim \frac{\langle X\rangle^{-\lambda_{\varrho,q-1}} }{\sqrt{c_{\lambda_{\varrho,q-1}}}}  \frac{dx}{\rho} \wedge \rho^{-\lambda_{\varrho,q-1}}v^{q-1} = \frac{\epsilon^{-\lambda_{\varrho,q-1}}}{\sqrt{c_{\lambda_{\varrho,q-1}}}} \ang X^{-2\lambda_{\varrho,q-1}}  \frac{dX}{\ang X}\wedge v^{q-1}.
$$ 
Since
$$
     \int_{-\infty}^{\infty}  \frac{\epsilon^{-\lambda_{\varrho,q-1}}}{\sqrt{c_{\lambda_{\varrho,q-1}}}} \ang X^{-2\lambda_{\varrho,q-1}}  \frac{dX}{\ang X}= \epsilon^{-\lambda_{\varrho,q-1}}\sqrt{c_{\lambda_{\varrho,q-1}}},
 $$   
this means that, in terms of cohomology classes,
\begin{equation}
[\omega_{\epsilon}^q]\sim \varepsilon^{-\lambda_{\varrho,q-1}}\sqrt{c_{\lambda_{\varrho,q-1}}} \; \pa_{q-1}\left[ v^{q-1} \right]  \quad \mbox{as} \; \varepsilon\searrow 0.
\label{cm.2}\end{equation}
Similarly, for $q=n+1$, if $v^{n}_+\in \mu_+^n$, let $\omega^{n+1}_\epsilon$ be the family of $L^2$-harmonic forms with respect to $\eth_{\epsilon}$ such that $\rho^n\omega^{n+1}_\epsilon\in \mu^{n+1}_{X_s}$ restricts to 
\begin{equation}
 \frac{1}{\sqrt{c_{\lambda_{\varrho,n}^+}}}\begin{pmatrix}0 \\ \langle X\rangle^{-\lambda_{\varrho,n}^+}\end{pmatrix}\rho^{n-\lambda_{\varrho,n}^+}v^{n}_+= \frac{\rho^n \epsilon^{-\lambda^+_{\varrho,n}}}{\sqrt{c_{\lambda_{\varrho,n}^+}}} \ang X^{-2\lambda^{+}_{\varrho,n}} \frac{dX}{\ang X}\wedge v^n_+ \quad \mbox{at} \quad \epsilon=0.
 \label{dual.5b}\end{equation}
 Again, if $\epsilon^{\mu}w$ is a higher order term in the expansion of $\omega^{n+1}_{\epsilon}$ as $\epsilon\searrow 0$, then again the restriction $\rho^n w_{sb,P}$ of $\rho^n w$ to $\bhs{sb,P}$ is in $\ker D_b$, so that 
 \begin{multline}
\rho^n w_{sb,P}= \left( \begin{array}{c}  0 \\ \ang X^{-\lambda_{\varrho,n}^+}  \end{array} \right)\rho^{n-\lambda_{\varrho,n}^+} w_{n,P}^+  \\ + \left( \begin{array}{c} 0\\ \ang X^{\lambda_{\varrho,n}^+}   \end{array} \right)\rho^{n+\lambda_{\varrho,n}^+} w_{n,P}^- + \left( \begin{array}{c}  \ang X^{\lambda_{\varrho,n+1}} \\0  \end{array} \right)\rho^{n-\lambda_{\varrho,n+1}} w_{n+1,P} + a_{n+1,P}
\label{dual.6}\end{multline} 
with $w_{n,P}^{\pm}\in \cH_{\pm}^n(\mathfrak{n}_P;V)$,  $w_{n+1,P}\in \cH^{n+1}(\mathfrak{n}_P;V)$ and $a_{n+1,P}\in \ker_2 D_{b,P}$.  Hence, we have that 
\begin{multline}
 \epsilon^{\mu} w_{sb,P}= \epsilon^{\mu-\lambda_{\varrho,n}^+}\ang X^{-2\lambda_{\varrho,n}^+}\frac{dX}{\ang X}\wedge w_{n,P}^+  + 
 \epsilon^{\mu+\lambda_{\varrho,n}^+}\ang X^{2\lambda^+_{\varrho,n}}\frac{dX}{\ang X}\wedge w_{n,P}^-  \\
 + \epsilon^{\mu-\lambda_{\varrho,n+1}}w_{n+1,P}+ \epsilon^{\mu}\rho^{-n}a_{n+1,P}.
 \label{dual.7}\end{multline}
 Clearly, the last term does not contribute cohomologically.  Moreover, if $w^-_{n,P}$ is non-zero, we must have that $\mu>\lambda^+_{\varrho,n}$, otherwise $\rho^n\omega^{n+1}_{\epsilon}$ would not vanish at $\bhs{sm}$ as a section of $\Lambda^*({}^{\ed}T^*X_s)\otimes E_s$.  Since $\mu>0$ and $\lambda_{\varrho,n+1}<0$, we thus see that the other terms are cohomologically negligible compared to \eqref{dual.5b} as $\epsilon\searrow 0$.  Hence, proceeding as before, we see that in terms of cohomology classes,
\begin{equation}
[\omega_{\epsilon,+}^{n+1}]\sim \varepsilon^{-\lambda_{\varrho,n}^+}\sqrt{c_{\lambda_{\varrho,n}^+}} \; \pa_{n}\left[ v^n_+ \right]  \quad \mbox{as} \; \varepsilon\searrow 0.
\label{cm.3}\end{equation}

Now, since the representation $\varrho$ is self-dual when $n$ is odd and $\varrho^*\cong \varrho\circ \vartheta$ if $n$ is even, we see that the above results holds for the dual representation $\rho^*$ and the corresponding dual  vector bundle $E^*$.  Now, Poincaré duality gives a canonical isomorphism
$$
     \cH^n(T_P;E_P)_{\pm}^*\cong \cH^n(T_P;E_P^*)_{\mp}.
$$
Hence, applying Poincaré duality on $M$ and $Z$, we can dualize \eqref{cm.2}.  More precisely,
for $n+1\le q\le 2n$ and $v\in \mu^{q}_Z$, let $\omega^q_{\epsilon}$ be the family of $L^2$-harmonic forms of degree $q$ with respect to $\eth_{\epsilon}$ such that $\rho^n \omega^q_{\epsilon}\in \mu^q_{X_s}$ restricts to 
$$
 \frac{1}{\sqrt{c_{-\lambda_{\varrho,q}}}}\begin{pmatrix}\langle X\rangle^{\lambda_{\varrho,q}} \\ 0\end{pmatrix}\rho^{n-\lambda_{\varrho,q}}v \quad \mbox{at} \quad \epsilon=0.
$$
Then, the dual statement of \eqref{cm.2} is that asymptotically, in terms of cohomology classes,
\begin{equation}
[\omega_{\epsilon}^q]\sim   \frac{\epsilon^{-\lambda_{\varrho,q}}}{\sqrt{c_{-\lambda_{\varrho,q}}}} (\iota_Y^*\iota_1^*)^{-1}\left[v\right] \quad \mbox{as} \; \varepsilon\searrow 0.
\label{cm.4}\end{equation}

Finally, for $q=n$ and $v^n_-\in \mu^n_-$, let $\omega_{\epsilon,-}^n$ be a family of $L^2$-harmonic forms with respect to $\eth_{\varepsilon}$ such that $\rho^n\omega^n_{\epsilon,-}\in \mu^n_{X_s}$ restricts to 
\begin{equation}
\frac{1}{\sqrt{c_{\lambda_{\varrho,n}^+}}}\begin{pmatrix}\langle X\rangle^{-\lambda_{\varrho,n}^+} \\ 0 \end{pmatrix}\rho^{n+\lambda^+_{\varrho,n}}v^n_-\quad \mbox{at}\quad \epsilon=0.
\label{cm.4b}\end{equation}

In this case,  it is delicate to take the Poincaré dual of \eqref{cm.3}, since the Poincaré duality on $M$ behaves differently from the Poincaré duality on $T_P$.  On the other hand, computing the asymptotic behavior of the cohomology class $[\omega_{\epsilon,-}^n]$ is slightly more complicated.  This is due to the fact that in the asymptotic expansion of $\omega_{\epsilon,-}^n$ as $\epsilon\searrow 0$, there is a lower order term which in terms of $L^2$ norm becomes negligible, but in terms of cohomology, is of the same order as the top order term in the expansion. 

To compute this term in the expansion of $\omega^n_{\epsilon,-}$, recall first that on the single surgery space $X_s$, $\rho$ is a boundary defining function for $\bhs{sb}$ and $x_{sm}:= \frac{\varepsilon}{\rho}=\frac{1}{\ang X}$ is a boundary defining function for $\bhs{sm}$. Hence, we see from \eqref{cm.4b} that the restriction of $\rho^n\omega^n_{\epsilon,-}$ at $\bhs{sb}$ has a term of order $x^{\lambda^+_{\varrho,n}}_{sm}$ at $\bhs{sm}$, 
$$
       \rho^n\omega^n_{\epsilon,-}\sim x^{\lambda^+_{\varrho,n}}_{sm} f = \varepsilon^{\lambda^+_{\rho,n}}\widetilde{f}\quad \mbox{as} \quad x_{sm}\searrow 0,
$$   
where $\widetilde{f}= \frac{f}{\rho^{\lambda^{+}_{\varrho,n}}}$ and 
$f$  is a bounded polyhomogeneous section of $\Lambda^n(\Ed T^*X_s)\otimes E_s$ on $\bhs{sm}$.  In particular, since $\left. f\right|_{\bhs{sm}\cap \bhs{sb}}\ne 0$, we see that $\widetilde{f}$  is not in $L^2$.  Moreover, since $D_{\varepsilon}(\rho^n\omega^n_{\epsilon,-})=0$ and $D_{\varepsilon}$ commutes with $\varepsilon$, we see that $D_{0}\widetilde{f}=0$, where $D_0$ is the restriction of $D_{\varepsilon}$ to $\bhs{sm}$.  

On the other hand, if $\rho^n\omega^n_{\epsilon,-}$ has a term of the form $\rho^\mu(\log\rho)^{\ell} w$ in its expansion at $\bhs{sb,P}$ with $\mu>0$, $\ell\in \bbN_0$ and $w\in \cA_{\phg}(X_s;\Lambda^n(\Ed T^*X_s)\otimes E_s)$ a bounded polyhomogeneous section, then since $D_{\epsilon}$ commutes with $\epsilon$, we see that
$$
    \left.\ang X^{\mu}(\log \ang X)^{j} w\right|_{\bhs{sb}}\in \ker D_b, \quad j\in\{0,\ldots,\ell\},
$$ 
though it is not necessarily in $\ker_{L^2}D_b$.  In particular, $\left.\ang X^{\mu} w\right|_{\bhs{sb}}\in \ker D_b$.  From the indicial family \eqref{sm.22}, we see that  the elements of the kernel of $\ker D_b$ only involve powers of $\ang X$ and logarithmic term $\log \ang X$ in their expansion at $X=\pm\infty$, which implies that $\ell=0$.

Then, from \eqref{sm.17} and \eqref{sm.26b}, this means that
\begin{multline}
\left.\ang X^{\mu} w\right|_{\bhs{sb,P}}= \frac{1}{\sqrt{c_{\lambda_{\varrho,n-1}}}}\left(\begin{array}{c} 0 \\ \ang X^{-\lambda_{\varrho,n-1}}  \end{array} \right) \rho^{n-\lambda_{\varrho,n-1}} w_{P}' + a_{n,P}' \\
+\frac{1}{\sqrt{c_{\lambda_{\varrho,n}^+}}} \left(\begin{array}{c} \ang X^{\lambda_{\varrho,n}^+} \\ 0  \end{array} \right) \rho^{n-\lambda_{\varrho,n}^+} w^+_{P}  +  \frac{1}{\sqrt{c_{\lambda_{\varrho,n}^+}}}\left(\begin{array}{c} \ang X^{-\lambda_{\varrho,n}^+} \\ 0  \end{array} \right) \rho^{n+\lambda_{\varrho,n}^+} w^-_{P} 
\end{multline}
with $w_{P}'\in \cH^{n-1}(T_P;E_P),$ $a_{n,P}'\in \ker_2 D_{b,P}$ and $w^{\pm}_{P}\in \cH^n(T_P;E_P)_{\pm}$.  Clearly, the second term does not contribute cohmologically and can be forgotten.  Now, the term coming from $w_{P}'$ in the expansion of $\omega^n_{\epsilon,-}$ is asymptotically of the form
\begin{equation}
  \epsilon^{\mu}  \frac{1}{\sqrt{c_{\lambda_{\varrho,n-1}}}}\left(\begin{array}{c} 0 \\ \ang X^{-\lambda_{\varrho,n-1}}  \end{array} \right) \rho^{-\lambda_{\varrho,n-1}} w_{P,0}'=
  \frac{1}{\sqrt{c_{\lambda_{\varrho,n-1}}}}\epsilon^{\mu-\lambda_{\varrho,n-1}}\ang X^{-2\lambda_{\varrho,n-1}}\frac{dX}{\ang X} \wedge w_{P}',
\label{do.3}\end{equation}
so that at the cohomological level,
\begin{equation}
\left[\frac{1}{\sqrt{c_{\lambda_{\varrho,n-1}}}}\epsilon^{\mu-\lambda_{\varrho,n-1}}\ang X^{-2\lambda_{\varrho,n-1}}\frac{dX}{\ang X} \wedge w_{P}'\right]\sim \epsilon^{\mu-\lambda_{\varrho,n-1}}\sqrt{c_{\lambda_{\varrho,n-1}}} \pa_{n-1}\left[ w_{P}' \right] \quad \mbox{as} \quad \epsilon\searrow 0
\label{do.4}\end{equation}
using the same argument leading to \eqref{cm.2}.  
\begin{remark}
Notice that \eqref{do.4}  is not negligible in the expansion of $\omega^n_{\epsilon}$, but it is when we compared to \eqref{cm.2} since we assume $\mu>0$.  
\label{do.4b}\end{remark}

On the other hand, for the term coming from $w^-_{P}$,  it is of the form $\epsilon^{\mu}$ times a term comparable to $\rho^n\omega^n_{\epsilon,-}$, so will be negligible cohomologically in the limit $\epsilon\searrow 0$.  

Finally, if $w^+_{P}\ne 0$, we see that the term coming from $w^+_{P}$ in the expansion of $\omega^n_{\epsilon,-}$ is
$$  
 \frac{\epsilon^{\mu}}{\sqrt{c_{\lambda_{\varrho,n}^+}}} \ang X^{\lambda_{\varrho,n}^+}  \rho^{-\lambda_{\varrho,n}^+} w^+_{P,0} = \frac{\epsilon^{\mu-\lambda^+_{\varrho,n}}}{\sqrt{c_{\lambda_{\varrho,n}^+}}}w^+_{P,0} \quad \mbox{as} \quad \epsilon\searrow 0.
$$
 When we pull-back to $Z$, this gives at the cohomological level
\begin{equation}
  [\iota_Z^* \omega^n_{\epsilon,-}] \sim \frac{\epsilon^{\lambda^+_{\varrho,n}}}{\sqrt{c_{\lambda_{\varrho,n}^+}}}[v^n_-] +  \sum_{P\in \mathfrak{P}_{\Gamma}}\frac{\epsilon^{\mu-\lambda^+_{\varrho,n}}}{\sqrt{c_{\lambda_{\varrho,n}^+}}}[w^+_{P,0}]  \quad \mbox{as} \quad \epsilon\searrow 0.
\label{do.5}\end{equation}
However, since \eqref{be.2} is an isomorphism in cohomology, we must have that $\mu=2\lambda^+_{\varrho,n}$ and that
\begin{equation}
[\iota_Z^*\omega^n_{\epsilon,-}]  \sim \frac{\epsilon^{\lambda^+_{\varrho,n}}}{\sqrt{c_{\lambda_{\varrho,n}^+}}} \iota^*_Z w_M^n \quad \mbox{as} \quad \epsilon\searrow 0,
\label{do.6}\end{equation} 
where $w_M^n\in H^n(X(Y);E)$ is the unique cohomology class such that $\pr_-\circ \iota_Y^*(w^n_{M})= [v^n_-]$.

Combining \eqref{cm.2}, \eqref{cm.3}, \eqref{cm.4}, \eqref{do.4}, \eqref{do.6} and keeping in mind Remark~\ref{do.4b}, we finally obtain
\begin{equation}
\begin{aligned}
\FP_{\varepsilon=0} \log\left(  \prod_q [\mu^q_M| \omega^q]^{(-1)^q}\right)= & \frac{\kappa^{\varrho}_{\Gamma}}2 \sum_{q<n} (-1)^q\left[ \log c_{\lambda_{\varrho,q}} \right] \dim\cH^q(\mathfrak{n};V) \\
         & + \frac{\kappa^{\varrho}_{\Gamma}}2 \sum_{q>n} (-1)^q\left[ \log c_{-\lambda_{\varrho,q}} \right] \dim\cH^q(\mathfrak{n};V) \\  
         & + \frac{\kappa^{\varrho}_{\Gamma}(-1)^n\log c_{\lambda^+_{\varrho,q}}}{2}\dim \cH^n(\mathfrak{n};V). \end{aligned}
\label{cm.10}\end{equation}

This yields the following formula relating analytic torsion and Reidemeister torsion.

\begin{theorem}
If the irreducible representation $\varrho$ is such that Assumption~\ref{dc.3b} and \eqref{sm.25} hold,  then the analytic torsion of $(X,g_{X}, E,h_E)$ and the Reidemeister torsion of $(X(Y),E,\mu_X)$ are related by
\begin{equation}
\log T(X;E,g_{X},h_E)=  \log\tau(X(Y),E,\mu_X)-\frac{1}2 \log \tau(Z,E,\mu_Z) - \frac{\kappa^{\varrho}_{\Gamma}}2 \left( \alpha_{\varrho}  + \beta_{\varrho} \right)
\label{cm.11d}\end{equation}
where $\kappa^{\varrho}_{\Gamma}=\#\mathfrak{P}^{\varrho}_{\Gamma}$ is the number of connected components $T_P$ of $\pa\bX$ for which $H^*(T_P;E_P)$ is non-trivial,  
\begin{equation}
\alpha_{\varrho} := \frac12 \sum_{q\ne n} (-1)^{q}(2q+1)\sign(q-n) \log \left( 2|\lambda_{\varrho,q}| \right) \dim\cH^q(\mathfrak{n};V) 
\label{cm.11e}\end{equation}
and 
\begin{equation}
 \beta_{\varrho}:=\sum_{q,a}\sum_{b>0} (-1)^q \dim V_{q,a,b} \log\left( \frac{a+\sqrt{a^2+b^2}}{b} \right)\label{cm.11f}\end{equation}
 with $V_{q,a,b}$ the vector spaces occurring  in the decomposition \eqref{sm.38} for the parabolic subgroup $P=P_0$.  

\label{cm.11c}\end{theorem}
\begin{proof}
The formula follows by taking the finite part as $\epsilon\searrow 0$ of the right hand side of \eqref{cm.1} via \eqref{mc.2} and \eqref{cm.10} and by applying Theorem~\ref{rt.2} to the left hand side of \eqref{cm.1}.
\end{proof}
\begin{corollary}
If $n$ is odd, then the formula of Theorem~\ref{cm.11c} simplifies to 
\begin{equation}
\log T(X;E,g_{X},h)=  \log\tau(X(Y),E,\mu_X) - \frac{\kappa^{\varrho}_{\Gamma}}2 \left( \alpha_{\varrho}  + \beta_{\varrho} \right)
\label{ff.1b}\end{equation}
with $\alpha_{\varrho}$ given more simply by
\begin{equation}
\alpha_{\varrho} :=  2\sum_{q< n} (-1)^{q} (n-q) \log \left( 2\lambda_{\varrho,q} \right) \dim\cH^q(\mathfrak{n};V) 
\label{ff.1c}\end{equation}
\label{ff.1}\end{corollary}
\begin{proof}
This follows from Theorem~\ref{rt.2} and \eqref{sm.36}.
\end{proof}

Comparing our formula with \cite[Theorem~1.1]{Pfaff2017} gives the following formula for the analytic torsion of the cusps.

\begin{corollary}
If \eqref{dc.3} holds, then 
\begin{multline}
   \log T_{reg}(F_X,\pa F_X;E)= -\frac{1}2 \log \tau(Z,E,\mu_Z) -\frac{\kappa_{\Gamma}}2 \left( \alpha_{\varrho}  + \beta_{\varrho} \right) 
   +c(n)(\rank E) \operatorname{vol}(\pa F_X) \\- \kappa_{\Gamma}\frac{(-1)^n}{4}\log(\lambda^+_{\varrho,n}) \dim\cH^n(\mathfrak{n}, V) - \frac{\kappa_{\Gamma}}4\sum_{q\ne n} (-1)^q
\log|\lambda_{\varrho,q}| \dim\cH^q(\mathfrak{n};V)\end{multline}   
where $c(n)$ is defined in \cite[(15.10)]{Pfaff2017} and $\kappa_{\Gamma}=\# \mathfrak{P}_{\Gamma}$ is the number of cusps of $(X, g_X)$.  
\label{ff.2}\end{corollary}

To conclude this section, let us give a proof of Corollary~\ref{int.7b}.
\begin{proof}[Proof of Corollary~\ref{int.7b}]
In this case, $\kappa_{\Gamma}^{\tau(m)}= \#\mathfrak{P}_{\Gamma}$ does not depend on $m$ and is simply the number of cusp ends of $X$.  By Corollary~\ref{ff.1} and \cite[Theorem~1.1]{MP2012}, the results will follow provided we can show that the defect $-\frac{\#\mathfrak{P}_{\Gamma}}2\left( \alpha_{\tau(m)}+ \beta_{\tau(m)}\right)$ in \eqref{ff.1b} is 
$\mathcal{O}(m^{\frac{n(n+1)}2}\log m)$ as $m\to \infty$.  To see this, recall that by Weyl's dimension formula, there exists a constant $C>0$ such that 
$$
    \dim(\tau(m))= Cm^{\frac{n(n+1)}2}+ \mathcal{O}(m^{\frac{n(n+1)}2-1}) \quad \mbox{as} \; m\to \infty.  
$$
Hence, we easily see that 
$$
  \dim \cH^*(\mathfrak{n};\tau(m))=\mathcal{O}(\dim \tau (m))= \mathcal{O}(m^{\frac{n(n+1)}2}) \quad \mbox{as} \; m\to \infty.  
 $$
 Thus, since by definition $\lambda_{\tau(m),q}= \tau_{q+1}+m+n-q$, we see directly from \eqref{ff.1c} that 
 $$
      \alpha_{\tau(m)}= \mathcal{O}(m^{\frac{n(n+1)}2}\log m) \quad \mbox{as} \; m\to \infty.  
 $$
 Similarly, $(\ker K_P)^{\perp}$ in \eqref{sm.38} is clearly such that 
 $$
   \dim (\ker K_P)^{\perp}=\mathcal{O}(\dim \tau (m))= \mathcal{O}(m^{\frac{n(n+1)}2}) \quad \mbox{as} \; m\to \infty, 
 $$
 while the weights $w_i$ in \eqref{dc.10} are all $\mathcal{O}(m)$ as $m\to \infty$.  Indeed, the $w_i$ are obtained by restricting the weights of $\varrho$ to $\mathfrak{a}$, so this follows from the description of the weights of $\varrho$ in terms of the highest weight, see for instance \cite[\S3.2.2]{GW}.  Therefore, we see from \eqref{cm.11f} that 
 $$
     \beta_{\tau(m)}= \mathcal{O}(m^{\frac{n(n+1)}2}\log m) \quad \mbox{as} \; m\to \infty,
 $$
 from which the result follows.  
\end{proof}

\section{Examples in dimension 3} \label{ed3.0}

We will now apply focus on the case $d=3$ and $n=1$ with $G=\SL(2,\bbC)$ and $K=\SU(2)$.  The standard Iwasawa decomposition of $G=NAK$ is then given by 
$$
     N= \left\{  \left( \begin{array}{cc}  1 & z \\ 0 & 1  \end{array} \right)\in \SL(2,\bbC) \; |  \; z\in \bbC \right\}
$$
with Lie algebra
$$
    \mathfrak{n}=  \left\{  \left( \begin{array}{cc}  0 & z \\ 0 & 0  \end{array} \right)\in \mathfrak{sl}(2,\bbC) \; |  \; z\in \bbC \right\}\cong \bbC
$$
and 
$$
     A= \left\{  \left( \begin{array}{cc}  a^{\frac12} & 0 \\ 0 & a^{-\frac12}  \end{array} \right)\in \SL(2,\bbC) \; |  \; a>0 \right\}
$$
with Lie algebra
$$
\mathfrak{a}= \left\{  \left( \begin{array}{cc}  \frac{\alpha}2 & 0 \\ 0 & -\frac{\alpha}2  \end{array} \right)\in \mathfrak{sl}(2,\bbC) \; |  \; \alpha\in \bbR \right\}\cong \bbR. 
$$  
In this case, the generator $H_1$ of $\mathfrak{a}$ is explicitly given by 
$$
    H_1= \left( \begin{array}{cc}  \frac{1}2 & 0 \\ 0 & -\frac{1}2  \end{array} \right).
$$
The standard parabolic subgroup with respect to this decomposition is then given by $P_0= NAM$ with 
$$
      M=\left\{  \left( \begin{array}{cc}  e^{i\theta} & 0 \\ 0 & e^{-i\theta}  \end{array} \right)\in \SL(2,\bbC) \; |  \; \theta\in \bbR \right\}.
$$
Now, the irreducible real representations of $\SL(2,\bbC)$ are given by the complex symmetric powers of the standard representation, 
$$
   \varrho_m: \SL(2,\bbC)\to \GL(V_m), \quad V_m:= \Sym^m(\bbC^2), \quad m\in \bbN.
$$
If we let $e_1=\left( \begin{array}{c}  1 \\ 0   \end{array} \right)$ and $e_2=\left( \begin{array}{c}  0 \\ 1   \end{array} \right)$ be the standard basis of $\bbC^2$, then 
$$
      v_j= e_1^j e_2^{m-j}, \quad j\in \{0,1,\ldots,m\}
$$
is a basis of $V_m$.  An admissible inner product on $V_m$ is obtained by declaring the basis $\{v_j\}$ to be orthonormal.   Regarding $dz$ and $d\bz$ as elements of $\mathfrak{n}^*\otimes \bbC$, a simple computation shows that 
\begin{equation}
\cH^q(\mathfrak{n};V_m) =\left\{  \begin{array}{ll} \bbC v_m, & q=0, \\ \bbC v_m d\overline z \oplus \bbC v_0 dz,  & q=1 \\
       \bbC v_0 dz\wedge d\overline z, & q=2, \end{array}  \right.
\label{md.3}\end{equation}
with $\cH_+^1(\mathfrak{n};V_m)=\bbC v_m d\bz$ and $\cH_-^1(\mathfrak{n};V_m)= \bbC v_0 dz$.  Moreover, the number $\lambda_{\varrho_m,q}$ describing the action of $H_1\in \mathfrak{a}$ on these spaces are given by
\begin{equation}
        \lambda_{\varrho_m,0}= \frac{m}2+1, \quad \lambda_{\varrho_m,1}^+= \frac{m}2,  \quad \lambda_{\varrho_m,1}^-= -\frac{m}2, \quad \lambda_{\varrho_m,2}= -\frac{m}2-1.        \label{dt.1}\end{equation}
We also compute that 
\begin{equation}
\begin{gathered}
  d_{\mathfrak{n}}v_k = (m-k)v_{k+1} dz, \quad
  d_{\mathfrak{n}}v_k  dz=0, \\
  d_{\mathfrak{n}} v_k d\overline{z} = (m-k) v_{k+1}  dz \wedge  d\overline{z}, \quad
  d_{\mathfrak{n}} v_k \frac{i}2 dz \wedge d\overline{z} =0, \\
  d^*_{\mathfrak{n}} v_k =0, \quad
  d^*_{\mathfrak{n}} v_k  dz = 2kv_{k-1}, \\
  d^*_{\mathfrak{n}} v_k  d\overline{z} =0, \quad
  d^*_{\mathfrak{n}}v_k \frac{i}2 dz \wedge d\overline{z} = ik v_{k-1} d\overline{z},
\end{gathered}
\label{mdb.1}\end{equation}  
 so that the Kostant Laplacian $K_{P_0}^2 =(d_{\mathfrak{n}}+ d_{\mathfrak{n}}^*)^2$ is given by
 $$
         K_{P_0}^2(v_k (dz)^p\wedge(d\bz)^q)= 2(k+1-p)(m-k+p)v_k (dz)^p\wedge(d\bz)^q. 
$$ 
Moreover, we have that 
$$
     \ker d^*_{\mathfrak{n}}\cap \left( \Lambda^q\mathfrak{n}^*\otimes V_m \right)= \left\{ \begin{array}{ll} V_m, & q=0, \\
                                                                                           V_m \otimes d\bz\oplus \bbC v_0 \otimes  d\bz, & q=1, \\
                                                                                           \bbC v_0 \otimes dz\wedge d\bz, & q=2,   \end{array}  \right.
$$
so that 
$$
    V_{0,j-\frac{m}2+1,\sqrt{2(j+1)(m-j)}}= \bbC v_j,  \quad V_{1,j-\frac{m}2, \sqrt{2(j+1)(m-j)}} \bbC v_j d\bz, 
$$
and otherwise $V_{q,a,b}=\{0\}$ for other values of $q$ and $a$ when $b>0$.

Now, let $\Gamma\subset \SL(2,\bbC)$ be a discrete subgroup such that $X=\Gamma\setminus G/K= \Gamma\setminus \bbH^3$ is a hyperbolic manifold of finite volume.  In this setting, the assumption \eqref{dc.3} does not necessarily hold, but \eqref{dc.3c} will. Furthermore, $\varrho_m(-\Id)= (-1)^m\Id$, so that \eqref{dc.3d} holds.   We can therefore apply Corollary~\ref{ff.1}, giving the following formula.

\begin{theorem}
For $d=3$, let $X= \Gamma \setminus \SL(2,\bbC) / \SU(2)$ a finite volume $3$-dimensional  hyperbolic manifold, where $\Gamma$ is a discrete subgroup of $\SL(2,\bbC)$.  Let $E\to X$ be the canonical bundle of \cite{MM1963} associated to the irreducible representation $\varrho_m: \SL(2,\bbC)\to \GL(V_m)$  and equipped with the admissible metric $h_E$.  Then we have the following relation between the analytic torsion of $(X,E,g_{X},h_E)$ and Reidemeister torsion of $(X(Y),E)$,
 $$
   \log T(X,E,g_X,h_E)= \log \tau(X(Y), E, \mu_{X}) - \kappa_m(X)\left(\log(m+2)+ \frac{B(m)}{2}\right)
 $$
 where 
 \begin{equation}
   B(m)=\sum_{\kappa=0}^{m-1}  \log\left( \frac{ \frac{m}2-\kappa +\sqrt{(\frac{m}2-\kappa)^2+2(1+\kappa)(m-\kappa)}}{\frac{m}2-\kappa-1+\sqrt{(\frac{m}2-\kappa-1)^2+2(1+\kappa)(m-\kappa)}} \right) 
 \label{cm.11b}\end{equation}
 and  $\kappa_m(X)$ is the number of connected  components $T_P$ of $\pa \bX$ for which $H^*(T_P;E_P)$ is non-trivial.  In particular, $\kappa_m(X)$ is equal to the number of cusps when $m$ is even, but can be smaller when $m$ is odd.   

\label{cm.11}\end{theorem}

\begin{corollary}
If $m$ is odd and $H^*(T_P;E_P)=\{0\}$ for each $P\in \mathfrak{P}_{\Gamma}$, then the formula simplifies  to
$$
 \log T(X,E,g,h)= \log \tau(\bX, E).
$$
\label{cm.13}\end{corollary}

In particular, when $\pa\bX=T_P$ is connected (\eg $X$ is the complement of a hyperbolic knot), we know from \cite{MFP2012} that $H^*(T_P;E_P)=\{0\}$ when $m$ is odd, so Corollary~\ref{cm.13} applies.    

Instead, when $m=2n$ is even, the formula gives.
\begin{corollary}
When $m=2\ell$ is even, this gives the following formula
$$
\log T(X,E,g,h)= \log \tau(\bX, E, \mu_{\bX}) + \kappa_{\Gamma}\log\left( b(\ell) \right).
$$
where
$$
  b(\ell):= \frac{1}{2\ell+2}\prod_{k=-\ell}^{\ell-1}\left( \frac{\sqrt{(\ell+1)^2+\ell^2-k^2}-k-1}{\sqrt{(\ell+1)^2+\ell^2-(k+1^2)}-k} \right)^{\frac12}.
$$
\label{cm.14}\end{corollary}
\begin{proof}
If we set $k= \kappa-\frac{m}2=\kappa-\ell$ in \eqref{cm.11b}, one computes that         
$$
B(2\ell)= -\log\left(  \prod_{k=-\ell}^{\ell-1}\left( \frac{\sqrt{(\ell+1)^2+\ell^2-k^2}-k-1}{\sqrt{(\ell+1)^2+\ell^2-(k+1^2)}-k} \right)\right),  
$$
from which the result follows.
\end{proof}
This should be compared with \cite[Theorem1.1]{Pfaff2014}.  In this formula, the torsion of $X$ is defined in terms of homology with basis specified in \cite{MFP2014}.  As one can check, the ratio of torsions considered in \cite[Theorem1.1]{Pfaff2014} remains the same if we define it instead in terms of cohomology using the bases specified in Theorem~\ref{cm.11}.  Hence, we deduce the following identity from Corollary~\ref{cm.14} and \cite[Theorem1.1]{Pfaff2014},
\begin{equation}
\frac{c(\ell)}{c(2)}= \frac{b(\ell)}{b(2)} \quad \mbox{for} \; \ell\ge 2,
\label{cm.15}\end{equation}
where 
$$
   c(\ell):= \frac{\prod_{j=1}^{\ell-1} (\sqrt{(\ell+1)^2+\ell^2-j^2}+\ell)  }{\prod_{j=1}^{\ell} (\sqrt{(\ell+1)^2+\ell^2-j^2}+\ell+1) }\left( \frac{\sqrt{(\ell+1)^2+\ell^2}+\ell }{\sqrt{(\ell+1)^2+\ell^2}+\ell+1 } \right)^{\frac12}
$$

\bibliographystyle{amsalpha}
\bibliography{torsion_cusp}

\end{document}